%% file: sn-article.tex
\documentclass[pdflatex,sn-mathphys-num]{sn-jnl}

\usepackage{graphicx}%
\usepackage{multirow}%
\usepackage{amsmath,amssymb,amsfonts}%
\usepackage{amsthm}%
\usepackage{mathrsfs}%
\usepackage[title]{appendix}%
\usepackage{xcolor}%
\usepackage{textcomp}%
\usepackage{manyfoot}%
\usepackage{booktabs}%
\usepackage{algorithm}%
\usepackage{algorithmicx}%
\usepackage{algpseudocode}%
\usepackage{listings}%

\usepackage{mathtools} 
\usepackage{xspace}
\usepackage[autostyle]{csquotes}
\usepackage[short]{optidef} 
\usepackage{tikz}
\usepackage{pgfplots}
\usetikzlibrary{shapes.geometric,patterns,arrows,positioning}

\DeclareMathAlphabet{\mathantt}{OT1}{antt}{li}{it}

\definecolor{myblue}{HTML}{1F77B4}
\definecolor{myorange}{HTML}{FF7F0E}
\definecolor{mygreen}{HTML}{2CA02C}
\definecolor{myred}{HTML}{D62728}

\newcommand*{\ie}{i.e.,\@\xspace}
\newcommand*{\eg}{e.g.,\@\xspace}
\newcommand*{\cf}{c.f.\@\xspace}

\DeclareMathOperator{\ran}{ran}
\newcommand*{\End}[1]{\ensuremath{ \mathrm{End} (#1) }}
\newcommand*{\rmspace}[2]{\ensuremath{ \mathrm{M} (#1 \times #2 , \mathbb{R}) }}
\newcommand*{\Matrix}[1]{\ensuremath{ \begin{pmatrix} #1 \end{pmatrix} }}
\newcommand*{\iprod}[2]{\ensuremath{ \langle #1, #2 \rangle }}
\newcommand*{\grass}[2]{\ensuremath{\mathcal{GR}(#1,#2)}}
\newcommand*{\stiefel}[2]{\ensuremath{\mathcal{ST}(#1,#2)}}
\newcommand*{\og}[1]{\ensuremath{\mathrm{O}({#1})}}
\newcommand*{\eqclass}[1]{\ensuremath{ [#1] }}
\newcommand*{\tspace}[2]{\ensuremath{T_{#1} #2}}

\newcommand*{\rmetric}[3]{\ensuremath{ \langle #1, #2 \rangle_{#3} }}
\DeclareMathOperator{\gradient}{grad}
\DeclareMathOperator{\hessian}{Hess}
\newcommand*{\ver}[2]{\ensuremath{ \mathrm{Ver}_{#1} #2 }}
\newcommand*{\hor}[2]{\ensuremath{ \mathrm{Hor}_{#1} #2 }}

\newcommand*{\gl}[1]{\ensuremath{\mathrm{GL}({#1})}}

\newcommand{\figref}[1]{\hyperref[#1]{Fig.~\ref*{#1}}}
\newcommand{\secref}[1]{\hyperref[#1]{Section~\ref*{#1}}}
\newcommand{\probref}[1]{\hyperref[#1]{Problem~\ref*{#1}}}
\newcommand{\theoref}[1]{\hyperref[#1]{Theorem~\ref*{#1}}}
\newcommand{\assref}[1]{\hyperref[#1]{Assumption~\ref*{#1}}}
\newcommand{\defref}[1]{\hyperref[#1]{Definition~\ref*{#1}}}
\newcommand{\lemmaref}[1]{\hyperref[#1]{Lemma~\ref*{#1}}}
\newcommand{\propref}[1]{\hyperref[#1]{Proposition~\ref*{#1}}}
\renewcommand{\algref}[1]{\hyperref[#1]{Algorithm~\ref*{#1}}}
\newcommand{\cororef}[1]{\hyperref[#1]{Corollary~\ref*{#1}}}
\newcommand{\appref}[1]{\hyperref[#1]{Appendix~\ref*{#1}}}

\theoremstyle{thmstyleone}%
\newtheorem{theorem}{Theorem}
\newtheorem{proposition}[theorem]{Proposition}%
\newtheorem{lemma}[theorem]{Lemma}
\newtheorem{corollary}[theorem]{Corollary}

\theoremstyle{thmstyletwo}%
\newtheorem{example}{Example}%
\newtheorem{remark}{Remark}%

\theoremstyle{thmstylethree}%
\newtheorem{problem}{Problem}

\begin{document}

\title[Shaping the Koopman dictionary by learning on the Grassmannian]{Shaping the Koopman dictionary 
by learning on the Grassmannian}


\author*[1]{\fnm{Roland} \sur{Schurig}}\email{roland.schurig@iat.tu-darmstadt.de}

\author[2]{\fnm{Pieter} \spfx{van} \sur{Goor}}\email{p.c.h.vangoor@utwente.nl}

\author[3]{\fnm{Karl} \sur{Worthmann}}\email{karl.worthmann@tu-ilmenau.de}

\author*[1]{\fnm{Rolf} \sur{Findeisen}}\email{rolf.findeisen@iat.tu-darmstadt.de}

\affil[1]{\orgdiv{Control and Cyber-Physical Systems Laboratory}, \orgname{Technische Universität Darmstadt}}

\affil[2]{\orgdiv{Robotics and Mechatronics group, Faculty of Electrical Engineering,  Mathematics, and Computer Science (EEMCS)}, \orgname{University of Twente}}

\affil[3]{\orgdiv{Optimization-based Control Group, Institute of Mathematics}, \orgname{Technische Universität Ilmenau}}


\abstract{
    Extended dynamic mode decomposition (EDMD) is a powerful tool to construct linear predictors of nonlinear dynamical systems by approximating the action of the Koopman operator on a subspace spanned by finitely many observable functions. 
    However, its accuracy heavily depends on the choice of the observables, which remains a challenge. 
    We propose a systematic framework to identify and shape observable dictionaries, reduce projection errors, and achieve approximately invariant subspaces. 
    To this end, we leverage optimisation on the Grassmann manifold and exploit inherent geometric properties for computational efficiency. 
    Numerical results demonstrate improved prediction accuracy and efficiency. 
    In conclusion, we propose a novel approach to efficiently shape the Koopman dictionary using differential-geometric concepts for optimisation on manifolds resulting in enhanced data-driven Koopman surrogates for nonlinear dynamical systems. 
}




\maketitle

\section{Introduction}\label{sec:introduction}

Data-driven techniques for modeling, analysis, and control of nonlinear systems have attracted considerable attention in recent years; see, e.g., \cite{FaulOu23,MartScho23,nguyen2023lmi}.
A method which has proven particularly useful for complex dynamical systems is extended dynamic mode decomposition (EDMD~\cite{Williams2015}) in the Koopman framework; see, e.g., \cite{rowley2009spectral,mezic2013analysis,MaurSusu20} and the recent overview paper~\cite{strasser2025overview} with a particular focus on controller design with closed-loop guarantees. 
B.O.\ Koopman's key idea~\cite{Koop31} is to replace the nonlinear dynamics by linear dynamics of real- or complex-valued functions called observables. 
Herein, the observables are mapped using the linear, but infinite-dimensional Koopman operator.  
Hence, a data-driven surrogate is computed on a lifted state space using EDMD, an efficient numerical algorithm to approximate the action of the Koopman operator on a given subspace spanned by finitely many observables, such that the dynamical system can be analysed using tools from linear systems theory, see~\cite{BrunKutz22}.

In the infinite-observable and infinite-data limit, it was shown that the Koopman operator is recovered~\citep{Korda2018}, see also~\cite{colbrook2022foundations} and~\cite{colbrook2024limits}.
Further, data-driven dictionaries are used in kernel EDMD~\cite{klus2020kernel,PhilScha24:kernel}. 
Then, Koopman invariance of suitably chosen reproducible kernel Hilbert spaces generated, e.g., by Wendland kernels~\cite{wendland2004scattered}, can be leveraged to derive uniform bounds on the full approximation error as first shown in~\cite{KohnPhil24}. 
However, such an invariance cannot be expected for finitely spanned subspaces. 
A possible remedy are eigenfunctions; see, e.g., \cite{mezic2005spectral,mauroy2016global,Brunton2016} and~\cite{mezic2020spectrum} for an in-depth treatment, see also~\cite{KordMezi20} for their optimal construction. However, in practice, observables typically correspond to sensors and the respective data samples are real-world measurements such that this approach may be prohibitive. 
Then, reprojection techniques as proposed in~\citep{GoorMaho23,van2025maximum} are required to ensure consistency of the lifted state. 

In conclusion, the number of steps, which can be reliably predicted in an auto-regressive manner, heavily depends on the choice of the observables. 
Our goal is to determine an approximately Koopman-invariant subspace encompassing an a-priori given set of observables, e.g., some sensors/observables of particular interest.
However, while adding further observables seemingly reduces a potential lack of invariance, the risk of overfitting increases. 
On the one hand, this may deteriorate the generalization capacity, i.e., the approximation quality on parts of the state space not present in the training data. 
On the other hand, the required amount of data in the EDMD algorithm quadratically grows in the number of observables, see, e.g., \cite{Williams2015,Zhang2023,NuskPeit23}.
Furthermore, removing obsolete observables may lead to improved representation of system-inherent properties, e.g., translational or rotational invariances in robotics~\cite{rosenfelder2024data}. 
Hence, we aim at resolving the trade-off between invariance according to the training data, maintaining numerical efficiency, and representation of key characteristics of the underlying dynamical system.
The latter may, e.g., indicate the choice of suitable classes of observables, e.g., monomials or Fourier modes. 

We partition the set of observables into two sets: 
The first contains the designated observables of interest, e.g., the full state in control or a particular set of sensors. 
The second set, which may be determined by deep-learning~\cite{Lusch2018,yeung2019learning,johnson2025heterogeneous}, provides additional observables that may be added to improve the prediction capacity by rendering the overall subspace (approximately) invariant. 
To this end, we set up an optimisation problem over \emph{subspaces} of a given (dictionary) size and derive an objective function encoding invariance as performance measure. 
Then, using differential-geometric arguments, we show that we may optimize on the Grassmann (rather than on the Stiefel) manifold to provide a trust-region algorithm to efficiently solve the rank-constrained optimisation problem. 
We illustrate the effectiveness of the proposed approach by exemplarily solving the task of rendering a dictionary approximately Koopman invariant by determining the best selection of an a-priori fixed number of observables out of a large class of potential observable functions. 
Hereby, we successfully incorporate differential-geometric arguments and tools to render the proposed approach comprehensable such that structural insights can be inferred from the chosen set of observables. 
A key feature of the proposed optimisation technique on the Grassmanian is its flexibility: It should, \eg also be applicable to more elaborated objective functions encoding completely different optimisation criteria. 
\\

\noindent \textbf{Related work.} In linear inverse modeling (LIM), similar questions were already asked and led to the development of optimal mode decomposition, see, e.g., \cite{kwasniok2022linear} and the references therein. Herein, the \textit{size of the subspace to be optimized} is penalized via a regularization term as originally proposed in~\cite{goulart2012optimal,wynn2013optimal} and successfully applied in~\cite{yang2020optimal}, see also the preprint~\cite{mieg2025optimal} for a recent extension of optimal mode decomposition towards systems with input analogously to EDMD with control as proposed in~\cite{Brunton2016}.
In~\cite{newton2023manifold}, the manifold-optimisation problem was tackled using a gradient-descent algorithm with less emphasize on the underlying structure and without a clear link to Koopman operator theory. A similar comment applies to~\cite{sashittal2018low}. 
However, to the best of our knowledge, the combination of differential-geometric techniques (optimisation on the Grassmannian) for subspace optimisation and a proper link to regression problem EDMD embedded in Koopman framework is not yet fully explored.
\\

\noindent The paper is organized as follows: \secref{sec:koopman} reviews state prediction in the Koopman framework and clearly introduces why we frame the prediction problem as one over \emph{subspaces}, and not bases. 
The research problem is defined in \secref{sec:invariant_subspace}, and \secref{sec:subspace_edmd} relates extended dynamic mode decomposition in the Koopman framework to the research problem.
\secref{sec:optimal_subspace} presents an algorithm for identifying key dictionary elements, while \secref{sec:grassmann_optimisation} provides the geometric foundations for implementing the algorithm. Simulation results are shown in \secref{sec:simulation}, followed by conclusions in \secref{sec:conclusion}.
\\

\noindent\textbf{Notation.}
Let $\mathbb{R}$ denote the field of real numbers. 
The vector space of real $(n \times m)$-matrices is $\rmspace{n}{m}$. 
We assume that $\rmspace{m}{n}$ is additionally equipped with the inner product $\iprod{A}{B} \coloneqq \operatorname{tr}(A^\top B)$, which induces the (Frobenius) norm $\| A \| \coloneqq \sqrt{\iprod{A}{A}}$.
For $n=m$, the subset of all invertible $(n \times n)$-matrices is the general linear group $\gl{n}$. 
If $V$ is any real vector space, the real vector space of all linear maps on $V$ is $\End{V}$. 
We use the Einstein summation convention to streamline notation, \ie we write simplified $x^i v_i = \sum_{i=1}^n x^i v_i$.
For any map $f \colon M \to N$, its restriction to $U \subset M$ is $\left. f \right|_U$.

\section{The Koopman Framework}
\label{sec:koopman}

We consider discrete-time autonomous dynamical systems
\begin{align} \label{eq:sys}
    x(t+1) &= f(x(t)), & 
    x(0) &= x_0,
\end{align}
with $f \colon \mathbb{R}^n \rightarrow \mathbb{R}^n$. 
Analogously to~\cite{Zhang2023}, let~$X \subseteq \mathbb{R}^n$ be non-empty, compact and forward invariant w.r.t.\ the dynamics~\eqref{eq:sys} to streamline the presentation, see~\cite{KohnPhil24} for a precise treatment of the general case. 
For each initial state $x_0 \in X$, the corresponding state trajectory $x(\cdot,x_0) \colon \mathbb{N}_0 \to X$ is iteratively generated by
\begin{align*}
    x(0,x_0) &= x_0, &
    &\textnormal{and} &
    \forall t \in \mathbb{N}_0 \colon x(t+1,x_0) = f(x(t,x_0)) .
\end{align*}
The transition map $f$ may be unknown, but it is possible to obtain a trajectory for each initial state by simulation or measurement. 

Let $H$ coincide with the Hilbert space~$L^2(X;\mathbb{R})$. We define the linear \emph{Koopman operator} $K : H \to H$ by the identity
\begin{equation*}
    \forall x \in X \colon (K \psi)(x) \coloneqq \psi (f(x))
\end{equation*}
for all $\psi \in H$.
The Koopman operator $K$ is linear but infinite dimensional.

Consider an $M$-dimensional subspace $V \subset H$ with $M < \infty$.
If $K \psi \in V$ for all $\psi \in V$, then we call the subspace $V$ \emph{Koopman invariant}, and we can represent the restriction $\left. K \right|_V$ of $K$ on $V$ exactly by using an $(M \times M)$-matrix. 
For a general finite-dimensional subspace $V \subset H$, even one that is not invariant, a \emph{compression} of $K$ on $V$ is defined as
\begin{equation*}
    C_V \coloneqq P_{V} \left. {K} \right|_{V} \in \End{V} ,
\end{equation*}
where $P_V : H \to V$ is a (not necessarily orthogonal) projection onto $V$.
Note that $C_V = K |_V$ if and only if $V$ is Koopman invariant.
If $V$ is not Koopman invariant, then the compression will be subject to a projection error, see~\cite{Zhang2023}.
For further details, we refer the reader to \appref{app:computing_compression}.

\subsection{Koopman Linear Systems} \label{sec:kls}

Assume that $B = (\psi_1, \dots, \psi_M)$ is a basis for $V$. 
Within the Koopman framework, this is often referred to as a \emph{dictionary}. 
Since $V$ has finite dimension, the compression $C_V \in \End{V}$ can be represented by its action on the elements in the basis $B$, which yields the coefficients
\begin{equation} \label{eq:compression_dictionary_basis}
    \forall\,j \in \{1, \dots, M \} \colon C_V \psi_j \eqqcolon a_j^k \psi_k \in V 
\end{equation}
Then, tacitly identifying $\mathbb{R}^M \cong \rmspace{M}{1}$, the \emph{matrix $K_B$ of $C_V$ relative to $B$} is
\begin{equation*}
    K_B = 
    \Matrix{
        a_1^1 & \dots & a_M^1 \\
        \vdots & \ddots & \vdots \\
        a_1^M & \dots & a_M^M 
    } .
\end{equation*}
This means that for any $\psi \in V$, the vector $C_V \psi \in V$ can be obtained by first expanding $\psi = v^j \psi_j \in V$ and then computing $w = (w^1, \dots, w^M) = K_B v \in \mathbb{R}^M$, which gives $C_V \psi = w^j \psi_j$.

As a next step, we construct the \emph{lift relative to $B$} as
\begin{equation} \label{eq:original_lift}
	\Psi \colon \begin{cases*}
		\mathbb{R}^n \to \operatorname{im} \Psi \\
		x \mapsto \Psi(x) \coloneqq (\psi_1(x), \dots, \psi_M(x))
	\end{cases*}
\end{equation}
Using the lift, for each $x_0 \in X$ we obtain the approximation
\begin{equation*}
    (\Psi \circ f)(x_0) 
    = \Matrix{
        (K \psi_1)(x_0) \\ \vdots \\ (K \psi_M)(x_0)
    }
    \approx \Matrix{
        (C_V \psi_1)(x_0) \\ \vdots \\ (C_V \psi_M)(x_0)
    } = K_B^\top \Psi(x_0),
\end{equation*}
with equality if and only if $V$ is Koopman invariant. 
The approximation quality depends on the invariance of $V$ with respect to $K$.
Suppose that $K_B^\top \Psi(x_0)$ is \enquote{close} to $(\Psi \circ f_0)(x_0) = \Psi(x(1,x_0))$. Then, we further approximate $\Psi(x(2,x_0)) = (\Psi \circ f)(f(x_0))$ with $K_B^\top (K_B^\top \Psi(x_0)) = (K_B^\top)^2 \Psi(x_0)$; and so on, inductively. 

The \emph{Koopman linear system relative to $B$} is thus defined as
\begin{align} \label{eq:kls}
    z(t+1) &= K_B^\top z(t), &
    z(0) &= \Psi(x_0) .
\end{align}
Note that, while the numerical values of $K_B$ depend on the chosen basis $B$, the value of $C_V \psi$ is invariant, and hence the system \eqref{eq:kls} is equivariant to the choice of basis. 
Given the Koopman linear system w.r.t.\ $B$, we denote the resulting trajectory by $z_B(\cdot, x_0)$ for $x_0 \in X$ to highlight its dependence on $B$, and speak of the \emph{Koopman prediction relative to~$B$}. 

Suppose that $E = (\varphi_1, \dots, \varphi_M)$ is another basis of the subspace~$V$. 
For each $j \in \{1,\dots,M\}$, we may uniquely expand $\varphi_j = p_j^k \psi_k$, and collect the components in the invertible \emph{change-of-basis matrix from $B$ to $E$}
\begin{equation} \label{eq:change_of_basis_matrix}
    P \coloneqq \Matrix{
        p_1^1 & \dots & p_1^M \\
        \vdots & \ddots & \vdots \\
        p_M^1 & \dots & p_M^M
    } \in \gl{M},
\end{equation}
for which the following statement holds, see \appref{app:compression_change_of_basis} for a proof.
\begin{lemma} \label{lemma:koopman_predicitions_basis}
    Let $V$ be a finite-dimensional subspace of $H$.
    Let $B = (\psi_1, \dots, \psi_M)$ and $E = (\varphi_1, \dots, \varphi_M)$ be two bases for $V$. 
    Let $P \in \gl{M}$ be the change-of-basis matrix from $B$ to $E$.
    Then, for each $x_0 \in X$, the Koopman predictions $z_B(\cdot,x_0)$ relative to $B$ and $z_E(\cdot, x_0)$ relative to $E$ satisfy $z_E(\cdot,x_0) = P z_B(\cdot, x_0)$.
\end{lemma}

\subsection{State Predictions in the Koopman Framework} \label{sec:state_predicitions_koopman}

We wish to make predictions in the \emph{state space}: Let $x_0 \in X$ be arbitrary and consider a basis $B = (\psi_1, \dots, \psi_M)$ for $V$, together with the lift $\Psi \colon \mathbb{R}^n \to \operatorname{im} \Psi$ relative to $B$.
After even a single timestep, unless the subspace $V$ spanned by $B$ is Koopman invariant, then there is no guarantee that $\Psi(\hat{x}_1) = K_B^\top \Psi(x_0)$ admits a solution $\hat{x}_1 \in \mathbb{R}^n$. 
Hence, it is not clear how $K_B^\top \Psi(x_0)$ should be mapped back to a state. 
To circumvent this problem, let $\pi_\Psi \colon \mathbb{R}^M \to \operatorname{im} \Psi$ be such that its restriction to $\operatorname{im} \Psi \subset \mathbb{R}^M$ is the identity map. 
We refer to \cite{GoorMaho23,van2025maximum} for a more elaborate discussion.

Assuming that $\Psi$ is one-to-one, an approximation $\hat{x}(\cdot, x_0)$ for $x(\cdot,x_0)$ can then be obtained as
\begin{equation} \label{eq:state_estimate}
    \forall t \in \mathbb{N}_0 \colon \hat{x}(t,x_0) \coloneqq (\Psi^{-1} \circ \pi_\Psi)(z_B(t,x_0)) .
\end{equation}
This is illustrated in \figref{fig:koopman_predictions}. 
\begin{figure}[]
    \centering
    \begin{tikzpicture}[>=latex, x=2cm,y=2cm]
        \begin{scope}[xshift=-3cm,x=2cm]
            \draw[->] (-1,0) -- (1,0);
            
            \draw (-0.9,2pt) -- node[below] (x0) {\footnotesize $x_0$} (-0.9,-2pt);
            \draw (0.81,2pt) -- node[below] (x1) {\footnotesize $x(1,x_0)$} (0.81,-2pt);
            \draw (0.2854,2pt) -- node[below] (xh1) {\footnotesize $\hat{x}(1,x_0)$} (0.2854,-2pt);
        \end{scope}
        \begin{scope}[xshift=3cm]
            \draw[->] (-1.25,0) -- (1.25,0);
            \draw[->] (0,-1.25) -- (0,1.25);
            
            \draw[mygreen] (0,0) circle [radius=1];
            
            \node at (0.6216,-0.7833)  (z0) [circle, draw, inner sep=1pt, fill=white] {};
            \node at ([xshift=3.75em,yshift=0.00em]z0) [] {\footnotesize $z_B(0,x_0) = \Psi(x_0)$};
            
            \node at (0.6545,0.1949)  (z1) [circle, draw, inner sep=1pt, fill=white] {};
            \node at ([xshift=-0.5em,yshift=0.75em]z1) [] {\footnotesize $z_B(1,x_0)$};
            
            \node at (0.6895,0.7243) (w1) [circle, draw, inner sep=1pt, fill=white] {};
            \node at ([xshift=2.25em,yshift=0.00em]w1) [] {\footnotesize $\Psi(x(1,x_0))$};
            
            \node at (0.9596,0.2815)  (zb1) [circle, draw, inner sep=1pt, fill=white] {};
            \node at ([xshift=2.5em,yshift=0.5em]zb1) [] {\footnotesize $\pi_{\Psi}(z_B(1,x_0))$};
            
            \node at (-0.625,0.5) [mygreen] {\footnotesize $\operatorname{im}(\Psi)$};
        \end{scope}
        \draw[->, black!50, dashed] ([yshift=3pt]x0.south) to[bend right] node[above] {\footnotesize $\Psi$} (z0);

        \draw[->, black!50, dashed] ([yshift=3pt,xshift=1pt]x1.north) to[bend left] node[above] {\footnotesize $\Psi$} (w1);
        \draw[->, dashed, black!50] ([yshift=3pt]x0.north) to[bend left] node[above,pos=0.5] {\footnotesize $f$} ([xshift=-1.5pt,yshift=3pt]x1.north);
        \draw[->,dashed,black!50] (zb1.west) to[bend left] node[above, pos=0.6] {\footnotesize $\Psi^{-1}$} (xh1.south);
        \draw[->, myorange, dashed] (z0) -- node[left] {\footnotesize $K_B^\top$} (z1);
        \draw[->,myblue, dashed] (z1) -- (zb1);
        \draw[myred, dotted] (z1) -- (w1);
    \end{tikzpicture}
    \caption{Estimation in lifted space for $n=1$, $M=2$ and $\Psi \colon x \mapsto (\cos(x), \sin(x))$.}
    \label{fig:koopman_predictions}
\end{figure}

Within the Koopman framework, a specific projection is often used, which is particularly suited for predictions in state space. 
This projection, however, only works for a certain class of finite-dimensional subspaces of $H$, which we introduce next.

For $k \in \{1, \dots, n \}$, introduce the $k$-th \emph{coordinate function} as
\begin{equation*}
    \theta_k \colon \begin{cases*}
        \mathbb{R}^n \to \mathbb{R} \\
        x = (x_1, \dots, x_n) \mapsto x_k 
    \end{cases*}
\end{equation*}
If a subspace $V \subset H$ contains all coordinate functions, we say that $V$ is a \emph{coordinate subspace}. 
These are the finite-dimensional subspaces we are working with. 
Usually, a basis for a coordinate subspace is chosen that contains all coordinate functions. 
However, in our algorithms, we do \emph{not} use such a basis for numerical reasons. 
Therefore, we introduce a projection that works for an arbitrary basis for a coordinate subspace.

Let $B = (\psi_1, \dots, \psi_M)$ be an arbitrary basis for a coordinate subspace $V$. 
Then, for each $k \in \{1, \dots, n\}$, there exist real numbers $c_k^j$ such that $\theta_k = c_k^j \psi_j$.
To define $\pi_\Psi$, we collect the components relative to the basis elements in $B$ in the matrix
\begin{equation*}
    \Pi_B \coloneqq \Matrix{
        c_1^1 & \dots & c_1^M \\
        \vdots & \ddots & \vdots \\
        c_n^1 & \dots & c_n^M
    } \in \rmspace{n}{m} .
\end{equation*}
We refer to $\Pi_B$ as the \emph{coordinate matrix relative to $B$}. 
Using $\Pi_B$, it can be shown that the lift $\Psi$ relative to $B$ is one-to-one, \cf \appref{app:coordinate_projection}.

Then, we define the \emph{coordinate projection} $\pi_\Psi \colon \mathbb{R}^M \to \operatorname{im} \Psi$ by
\begin{equation*}
    \pi_\Psi(z) \coloneqq \Psi(\Pi_B z) ,
\end{equation*}
which indeed satisfies $\left. \pi_\Psi \right|_{\operatorname{im} \Psi} = \mathantt{id}_{\operatorname{im} \Psi}$. 
We establish this in \appref{app:coordinate_projection} as well.

Consider again the Koopman linear system \eqref{eq:kls}.
From the above definition of the coordinate projection, we obtain the estimate $\hat{x}(\cdot, x_0)$ for $x(\cdot,x_0)$ from $z_B(\cdot,x_0)$ via the linear map
\begin{equation} \label{eq:kls_state_predicition}
    \hat{x}(\cdot, x_0) = (\Psi^{-1} \circ \pi_\Psi)(z_B(\cdot, x_0)) = \Pi_B z_B(\cdot, x_0) .
\end{equation}
\begin{remark}
    In the simple case where for each $k \in \{1, \dots, n \}$, the element $\psi_k$ of the basis $B$ is given by $\theta_k$, the coordinate matrix relative to $B$ is given by
    \begin{equation*}
        \Pi_B = \Matrix{I_n & 0} .
    \end{equation*}
    Hence, in this case for each $z \in \mathbb{R}^M$, the state $(\Psi^{-1} \circ \pi_{\Psi})(z) \in \mathbb{R}^n$ simply corresponds to the first $n$ components of $z$. 
    We call such a basis a \emph{coordinate basis} for $V$.
\end{remark}
The construction of $\Pi_B$ relies on the specific basis $B$ used to span $V$.
However, our previous result readily allows us to show that this state prediction is invariant under a change of basis, see \appref{app:coordinate_projection} for a proof.
\begin{corollary} \label{coro:state_prediction}
    Let $V \subset H$ be a finite-dimensional coordinate subspace of $H$. 
    Let $B = (\psi_1, \dots, \psi_M)$ and $E = (\varphi_1, \dots, \varphi_M)$ be two arbitrary bases for $V$. 

    Then, for each $x_0 \in X$, the state prediction $\Pi_B z_B(\cdot,x_0)$ based on the Koopman prediction relative to $B$ and the state prediction $\Pi_E z_E(\cdot,x_0)$ based on the Koopman prediction relative to $E$ coincide.
\end{corollary}
For this invariance reason, for each $x_0 \in X$ we refer to $\hat{x}(\cdot, x_0)$ as the \emph{state prediction on $V$ (from $x_0$)}.

To conclude this section, while for each initial state $x_0 \in X$ the predictions $z_B(\cdot,x_0)$ and $z_E(\cdot,x_0)$ in the \emph{lifted space}, obtained from different bases $B$ and $E$ of the same coordinate subspace $V \subset H$, do depend on the chosen basis, there exists a {canonical} way to obtain predictions in state space.
This prediction relies on the matrix that relates the basis elements to the coordinate functions, the use of which renders the state prediction basis-invariant.

\section{Invariant Subspaces for Prediction}
\label{sec:invariant_subspace}

We are now prepared to introduce ideas for shaping a Koopman dictionary.
For this, let $V \subset H$ be a finite-dimensional coordinate subspace. 
Our goal is to find a proper coordinate subspace $W \subset V$ and then apply the Koopman framework for basis-independent prediction from \secref{sec:state_predicitions_koopman} \emph{to $W$ instead of $V$}. 
Ideally, by passing from $V$ to $W$, we exclude the components from $V$ that are irrelevant to the system's dynamics and geometry, e.g., due to translational invariances in non-holonomic systems, see~\cite{rosenfelder2024data}.

Initially, we make $V$ as \enquote{rich} as possible by selecting a basis~$B$ of the high-dimensional subspace~$V$, \eg by relying on a certain class of basis functions. 
We do have in mind that possibly not all elements in $B$ are required or even helpful. 
For some components of $V$, however, we might have some prior knowledge or intuition and wish to keep them in $W$ in any case.

To make that precise, suppose that $T$ is a coordinate subspace of $V$. 
The idea is to encode all prior knowledge we might have about the system into the subspace $T$ and to then ensure that $T \subset W$ holds while constructing $W$. 
Note that this will also ensure that $W$ is a coordinate subspace. 
The subspace $T$, however, might be far away from being Koopman invariant, which is why we aim to identify in the \enquote{remaining} components of $V$ the ones that best compensate for the projection error on~$T$, resulting from a lack of invariance.

To make that precise, we first write $V$ as the direct sum of $T$ and another subspace $R$ of $V$. 
Based on that, we present the following general framework.
\begin{enumerate}
    \item[(i)] Fix a dimension $0 < r < \dim R$.  
    \item[(ii)] Find a subspace $S \subset R$ of dimension $r$ such that for each $x_0 \in X$ the state prediction on $W \coloneqq T \oplus S \subset V$ from $x_0$ is \enquote{as good as possible}.
\end{enumerate}
The most important step is clearly (ii), where our design choice is the subspace $S$ of $R$, which in turn defines $W$.
The setup is illustrated in \figref{fig:sketch_subspaces}.
\begin{figure}[htb]
    \centering
    \begin{tikzpicture}[>=latex, x=1cm]
        \begin{scope}[xshift=-3.5cm]
            \draw (0,0) rectangle (5,2);
            \draw (2,0) -- (2,2);
    
            \node at (4.9,2.25) (V) {$V$};
            \node at (1,1) (T) {$T$};
            \node at (3.5,1) (R) {$R$};
        \end{scope}
        \begin{scope}[xshift=3.5cm]
            \draw (2,0) -- (5,0) -- (5, 1.25);
            \draw (2,1.25) -- (2,2);
            \draw[myorange] (0,0) -- (0,2) -- (5,2) -- (5,1.25) -- (2,1.25) -- (2,0) -- cycle;
    
            \node at (4.9,2.25) (V) {$V$};
            \node[myorange] at (0.2,2.25) (W) {$W$};
            \node at (1,1) (T) {$T$};
            \node at (3.5,1.625) (S) {$S$};
            \node at (3.5,0.625) (Q) {$Q$};
        \end{scope}
        \draw[myblue,->] (1.75,1) -- node[above] {\footnotesize Choice of} node[below] {\footnotesize $S \subset R$} (3.25,1);
    \end{tikzpicture}
    \caption{Sketch of $V = T \oplus R$ and $W = T \oplus S \subseteq V$, with $R = S \oplus Q$.}
    \label{fig:sketch_subspaces}
\end{figure}
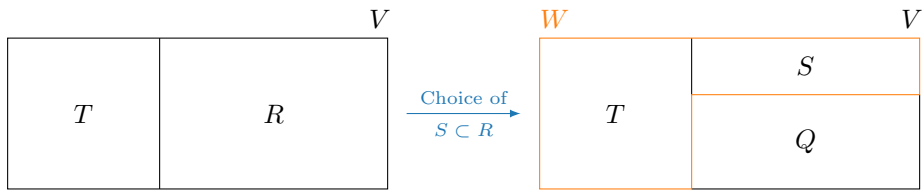

We now discuss how the predictive quality of the reduced-order subspace $W = T \oplus S$ can be evaluated.
Recall from \secref{sec:kls} that, by choosing \emph{any} basis for $W$, we can construct a Koopman linear system of the form \eqref{eq:kls} relative to that basis.
Using \cororef{coro:state_prediction}, we obtain \emph{basis-independent} state predictions from this basis-dependent Koopman linear system. 

The degree of freedom that we have in designing $W = T \oplus S$ is the choice of the $r$-dimensional subspace $S$, since we assume that $T$ is fixed.
To emphasise the dependence on the subspace $S$, we denote the state prediction on $W$ from $x_0 \in X$ by $\hat{x}(\cdot, x_0; S)$.

As with numerical solvers for initial value problems, one cannot generally expect a reasonable approximation accuracy for arbitrarily long horizons.
However, this may be possible on \emph{finite} horizons of length $N \in \mathbb{N}$, which is the focus of this work.

To measure the distance between two maps $\xi, \zeta$ from $\mathbb{N}_0$ to $\mathbb{R}^n$ on the finite horizon $N$, in this work we use the \emph{performance measure} $d_N$ defined by
\begin{equation} \label{eq:trajectory_distance}
    d_N(\xi,\zeta) \coloneqq \left\| \left( \delta(0), \delta(1), \dots , \delta(N) \right) \right\|_2 = \sum_{i=0}^N \| \delta(i) \|_2 = \sum_{i=0}^N \sqrt{\sum_{j=1}^n \delta_j(i)^2},
\end{equation}
with $\delta(\cdot) = \| \xi(\cdot) - \zeta(\cdot)  \|_2$. 
However, other performance measures tailored to specific needs are possible.

To evaluate how well the state prediction in the Koopman framework on the subspace $W = T \oplus S$ approximates the true solution, we are now prepared to introduce
\begin{equation} \label{eq:invariance_measure}
    \mu_N(S) \coloneqq \sup_{x_0 \in X} \{ d_N(x(\cdot,x_0), \hat{x}(\cdot, x_0; S)) \} .
\end{equation}
Note that, if $W = T \oplus S$ is Koopman invariant, then $\mu_N(S) = 0$. 
Otherwise, $\mu_N(S) \geq 0$. 
Hence, $\mu_N$ can be seen as a basis-invariant measure for the invariance of $W$, tailored to the specific task of predicting trajectories in state space.

Finally, we propose the following research problem.
\begin{problem}[Almost Invariant Subspace for Predicitions] \label{prob:invariant_subspace_predicition}
    Given a coordinate subspace $V = T \oplus R$ of $H$ as described above, a prediction horizon $N \in \mathbb{N}$, and a reduced dimension $r \in \mathbb{N}$, find an $r$-dimensional subspace $S \subset R$ such that $\mu_N(S)$ is minimal.
\end{problem}

\section{Computing Compressions from Data} \label{sec:subspace_edmd}

Problem~\ref{prob:invariant_subspace_predicition} can, in general, not be solved analytically to global optimality. 
Hence, in \secref{sec:optimal_subspace}, we propose an optimisation-based framework for finding (suboptimal) solutions. 
In this framework, the subspace $S$ is the \emph{decision variable} of the optimisation problem. 
To set up an objective function for such an optimisation problem, we need to establish a relation between a Koopman linear system \eqref{eq:kls} relative to a basis of $V$ with a Koopman linear system relative to a basis $W$. 
To this end, we set up a data-driven framework, in which this can be done: First, recall the extended dynamic mode decomposition algorithm proposed in~\cite{Williams2015} to compute a compression of the Koopman operator from data, see also \appref{app:computing_compression}, and~\cite{colbrook2024multiverse} for more sophisticated variants of EDMD.
Subsequently, we recall the decomposition of $V$ sketched in \figref{fig:sketch_subspaces} and find an expression for the compression on $W = V \oplus S$, which allows us to set up \probref{prob:invariant_subspace_predicition} as an optimisation problem over $S$.
More specifically, the derivations enable us to formulate an objective function that can be minimised to design the subspace $S$ in \secref{sec:optimal_subspace}. 
To make these optimisation routines efficient, we additionally introduce a basis for $W$ with appealing numerical properties.
Throughout this section, we fix a finite-dimensional subspace $V$ of $H$, and let $B = (\psi_1, \dots, \psi_M)$ be basis for $V$.

\subsection{Extended Dynamic Mode Decomposition}

For the extended dynamic mode decomposition algorithm, we collect \emph{training data}
\begin{equation} \label{eq:edmd_training_data}
    D = \{ (x_i, y_i = f(x_i)) \}_{i=1}^L \subset X \times X
\end{equation}
with $L \geq M$, and where the $x_i$ are drawn i.i.d. 

As a next step, we form the data matrices
\begin{align} \label{eq:edmd_matrices}
    {G}_B &= \Matrix{
        \psi_1({x}_1) & \dots & \psi_1({x}_L) \\
        \vdots & \ddots & \vdots \\
        \psi_M({x}_1) & \dots & \psi_M({x}_L)
    }, &
    S_B &= \Matrix{
        \psi_1({y}_1) & \dots & \psi_1({y}_L) \\
        \vdots & \ddots & \vdots \\
        \psi_M({y}_1) & \dots & \psi_M({y}_L)
    } ,
\end{align}
obtained from the training data $D$ and relative to the basis $B$.
We assume that $G_B$ defined in \eqref{eq:edmd_matrices} has full row rank.

In this case, it can be shown that
\begin{equation*}
    K_B \coloneqq (G_B G_B^\top)^{-1} G_B S_B^\top 
\end{equation*}
is the matrix representation of the EDMD compression
\begin{equation*}
    C_V \coloneqq P_V \left. K \right|_V \End{V} ,
\end{equation*}
where $P_V$ is a projection of $H$ with $\ran P_V = V$, and $\ker P_V$ is determined by the training data $D$. 
We refer to \appref{app:computing_compression} for the details.

Next, we consider a proper subspace $W \subset V$, with $\dim W = \ell$. 
Let $\tilde{B} = (\omega_1, \dots, \omega_\ell)$ be a basis for $W$. 
Because $B$ is a basis for $V$ and $W$ is contained in $V$, we can express the basis vectors of $W$ as a linear combination of the basis vectors of $V$. 
Specifically, we assume
\begin{equation*}
    \forall j \in \{1, \dots, \ell \} \colon \omega_j = \bar{u}_j^k \psi_k ,
\end{equation*}
and collect the components in the matrix
\begin{equation}
    \bar{U} = \Matrix{
        \bar{u}_1^1 & \dots & \bar{u}_\ell^1 \\
        \vdots & \ddots & \vdots \\
        \bar{u}_1^M & \dots & \bar{u}_\ell^M
    } \in \rmspace{M}{\ell} ,
\end{equation}
which has full column rank. 
It follows that the data matrices relative to $\tilde{B}$ are given by
\begin{align*}
    G_{\tilde{B}} &= \bar{U}^\top G_B, &
    S_{\tilde{B}} &= \bar{U}^\top S_B,
\end{align*}
so that the matrix representation $K_{\tilde{B}}$ of the EDMD compression on $W$ relative to $B$ is given by
\begin{equation*}
    K_{\tilde{B}} = (\bar{U}^\top G_B G_B^\top \bar{U})^{-1} \bar{U}^\top G_B S_B^\top \bar{U} ,
\end{equation*}
which has the disadvantage that every time the subspace $W$ is changed -- which will be the case in each step of our optimisation routine -- a different matrix has to be inverted.

To circumvent this problem, we introduce a new basis $E = (\varphi_1, \dots, \varphi_M)$ for $V$, which was also mentioned in \cite{Haseli2023}, and which has appealing numerical properties. 
This new basis will depend both on $B$ and the training data $D$ in \eqref{eq:edmd_training_data}.
The first step in constructing $E$ is to compute the thin QR decomposition of the data matrix $G_B^\top = \hat{Q} \hat{R}$ \cite[Theorem~5.2.3]{Golub2013}, and to use
\begin{equation} \label{eq:qr_transformation}
    P \coloneqq \hat{R}^{- \top} \in \gl{M} ,
\end{equation}
as the change-of-basis matrix from $B$ to $E$. 
This new basis enjoys the numerical advantage of being orthonormal in a desirable sense.
\begin{corollary} \label{coro:data_orthonormal_basis}
    Let $V$ be a finite-dimensional subspace of $H$, and suppose that $B = (\psi_1, \dots, \psi_M)$ is a basis for $V$. 
    Assume that the data matrix $G_B$ relative to $B$ has full row rank. 
    Let $E = (\varphi_1, \dots, \varphi_M)$ be the basis for $V$ such that $P$ defined by \eqref{eq:qr_transformation} is the change-of-basis matrix from $B$ to $E$.
    
    Then, the data matrix $G_E$ relative to $E$ satisfies $G_E G_E^\top = I_M$.
\end{corollary}
\begin{proof}
    We first ensure that $E$ is well-defined, \ie that $P \in \gl{M}$ exists. 
    Since $G_B$ has full row rank, its transpose $G_B^\top$ has full column rank. 
    This implies that $\hat{R}$ has full rank \cite[Theorem~5.2.3]{Golub2013}, so that $P$ does exist and has full rank.
    
    By construction of $E$, we have $G_E = P G_B$. 
    This gives
    \begin{equation*}
        G_E = P (\hat{R}^\top \hat{Q}^\top) = (P \hat{R}^\top) \hat{Q}^\top = (\hat{R}^{-\top} \hat{R}^\top) \hat{Q}^\top = \hat{Q}^\top
    \end{equation*}
    which implies $G_E G_E^\top = \hat{Q}^\top \hat{Q} = I_M$, since $\hat{Q}$ has orthonormal columns by construction.
\end{proof}

\subsection{Subspace Extended Dynamic Mode Decomposition} 

We now move on to the main idea from \secref{sec:invariant_subspace}, namely to choose a subspace $S \subset R$  and to consider the reduced-order subspace $W = T \oplus S \subset V$. 
We show how this idea can be incorporated into the general EDMD routine introduced above and explain why the reported change of basis from $B$ to $E$ is numerically convenient.

The idea with the two bases $B$ and $E$ is to first define $V$ and the coordinate subspace $T$ of $V$ by choosing a basis $B$. 
Without loss of generality, we may assume that $T$ is spanned by the first $s$ basis elements $(\psi_1, \dots, \psi_s)$. 
Then, we switch to the new basis $E$ for its numerical advantages.
To express $T$ in the new basis $E$, the following result is proven.
\begin{lemma} \label{lemma:equal_span}
    Let the assumptions from \cororef{coro:data_orthonormal_basis} hold.

    Then, for each $k \in \{1, \dots, M \}$, the span of $(\psi_1, \dots, \psi_k)$ and $(\varphi_1, \dots, \varphi_k)$ coincide.
\end{lemma}
\begin{proof}
    Fix any $k \in \{ 1, \dots, m \}$. 
    Denote by $T$ and $\hat{T}$ the span of $ (\psi_1, \dots, \psi_k)$ and $(\varphi_1, \dots, \varphi_k)$, respectively. 
    Recall from the proof of \cororef{coro:data_orthonormal_basis} that $P \in \gl{M}$ holds.
    
    We prove the statement by showing $T \subseteq \hat{T}$. 
    Then, equality follows from $\dim T = \dim \hat{T} = k$, which follows from the linear independence of both $(\psi_1, \dots, \psi_k)$ and $(\varphi_1, \dots, \varphi_k)$.
    
    To establish the desired inclusion, we consider the inverse change-of-basis matrix $P^{-1} = \hat{R}^\top$ from $E$ to $B$, which has the lower triangular form
    \begin{equation*}
        P^{-1} = \Matrix{
        q_1^1 \\
        q_2^1 & q_2^2 \\
        \vdots & \vdots & \ddots \\
        q_m^1 & q_m^2 & \dots & q_m^m
        } .
    \end{equation*}
    This is because $\hat{R}$ in \eqref{eq:qr_transformation} is upper triangular by construction, which implies that $\hat{R}^\top$ is lower triangular. Hence, for each $j \in \{ 1, \dots, k \}$ we have
    \begin{equation*}
        \psi_j = \sum_{i=1}^j q_j^i \varphi_i 
    \end{equation*}
    Now, let $\psi = \sum_{j=1}^k a^j \psi_j \in T$ 
    be arbitrary. 
    The above observation implies
    \begin{equation}
        \psi = \sum_{j=1}^k a^j \left( \sum_{i=1}^j q_j^i \varphi_i \right) = \sum_{j=1}^k b^j \varphi_j \in \hat{T},
    \end{equation}
    where for $j \in \{1, \dots, k \}$ we have $b^j = \sum_{i=j}^k q_i^j a^i$. 
    This gives $T \subseteq \hat{T}$, and the assertion follows.
\end{proof}
Therefore, also in the new basis $E$, the coordinate subspace $T \subseteq V$ is spanned by the first $s$ basis elements $(\varphi_1, \dots, \varphi_s)$. 
Introducing the subspace $R$ of $V$ as the one spanned by $(\varphi_{s+1}, \dots, \varphi_M) \eqqcolon (\alpha_1, \dots, \alpha_d)$ allows us to write $V$ as the direct sum of $T$ and $R$, \cf \figref{fig:sketch_subspaces}.

Further, following \probref{prob:invariant_subspace_predicition}, we now consider a \emph{fixed} subspace $S \subset R$ with $0 < \dim S = r < d$. 
In \secref{sec:optimal_subspace}, $S$ will be a decision variable that we optimise over.

Let $(\beta_1, \dots, \beta_r)$ be a basis for $S$. 
For computations, we represent the subspace $S$ with a matrix, which depends on the chosen basis. 
Since $S$ is a subspace of $R$, we can express each of its basis vectors as linear combinations of the basis vectors for $R$.
Assume, for $j \in \{1, \dots, r\}$, that they are given by
\begin{equation*}
	\beta_j = u_j^k \alpha_{k} ,
\end{equation*}
with $u_j^k \in \mathbb{R}$. 
Let us collect the components of $(\beta_1,\dots,\beta_r)$ relative to $(\alpha_{1}, \dots, \alpha_d)$ in the matrix
\begin{equation*}
    U = \Matrix{
        u_1^1 & \dots & u_r^1 \\
		\vdots & \ddots & \vdots \\
		u_1^d & \dots & u_r^d
    } \in \rmspace{d}{r} .
\end{equation*}
Without loss of generality, we may assume that $U$ has orthonormal columns, \ie $U^\top U = I_r$. 
We refer to \appref{app:ONB} for a proof.

Finally, we have all the results together to efficiently represent the EDMD compression on $W$. 
\begin{theorem} \label{theo:edmd_compression_reduced}
    Let the assumptions of \cororef{coro:data_orthonormal_basis} hold.
    For a fixed $s \in \{1, \dots, M-1\}$, denote by $T$ the subspace spanned by $(\varphi_1, \dots, \varphi_s)$, and by $R$ the subspace spanned by $(\varphi_{s+1}, \dots, \varphi_M)$. 
    
    Let $S$ be a subspace of $R$ of dimension $r \in \{1, \dots, M-s-1\}$, and suppose that $(\beta_1, \dots, \beta_r)$ is a basis for $S$ such that the matrix $U$ of components relative to $(\varphi_{s+1}, \dots, \varphi_M)$ has orthonormal columns. 
    Consider the subspace $W = T \oplus S$ of $V$, and the basis $\tilde{E} = (\varphi_1, \dots, \varphi_s, \beta_1, \dots, \beta_r)$ for $W$.
    
    Then, the matrix representation $K_{\tilde{E}}$ of the EDMD compression on $W$ relative to $\tilde{E}$ is given by
    \begin{equation*}
        K_{\tilde{E}} 
        = \bar{U}^\top G_E S_E^\top \bar{U} ,
    \end{equation*}
    with $\bar{U} = \operatorname{blkdiag}(I_s, U)$.
\end{theorem}
\begin{proof}
    First, note that by \lemmaref{lemma:edmd_stiefel_basis} a basis $(\beta_1, \dots, \beta_r)$ of the claimed form always exists. 
    Further, the sum of $T$ and $S$ is surely direct, since $T \cap R = \{ 0 \}$ and $S \subset R$ imply $T \cap S = \{0 \}$. 
    This also implies that $\tilde{E}$ is indeed a basis for $W$.
    
    From \cororef{coro:data_orthonormal_basis} we have $G_E G_E^\top = I_M$, so \theoref{theo:koopman_compression} implies that the matrix representation of the EDMD compression on $V$ relative to $E$ is $K_E = G_E S_E^\top$. 
    Note that $\bar{U}$ collects the components of $\tilde{E}$ relative to $E$.
    Thus, \cororef{coro:koopman_compression_sub} implies $K_{\tilde{E}} = (\bar{U}^\top \bar{U})^{-1} \bar{U}^\top G_E S_E^\top \bar{U}$, and the assertion follows from observing
    \begin{equation*}
        \bar{U}^\top \bar{U} = \Matrix{I_s \\ & U^\top} \Matrix{I_s \\ & U} = \Matrix{I_s \\ & U^\top U} = \Matrix{I_s \\ & I_r} = I_{s+r} .
    \end{equation*}
\end{proof}
Therefore, we insist on orthonormal columns for the matrix $U$  -- which represents the subspace $S \subset R$ -- in order to maintain the appealing orthonormality properties introduced to the EDMD algorithm through the choice of the \enquote{QR basis} $E$. 
Basically, we get away with inverting $G_B G_B^\top$ only once using the QR decomposition, and then exploiting this result by restricting our algorithm to work with an orthonormal basis for $S$, in the sense of \lemmaref{lemma:edmd_stiefel_basis}. 
We present the geometric structure for that in \secref{sec:grassmann_optimisation}.

The last step for the state prediction framework in \secref{sec:state_predicitions_koopman} is the coordinate matrix $\Pi_{\tilde{E}}$ relative to $\tilde{E}$. 
For this, we select $B$ as a coordinate basis, and demand that the subspace $T$ contains all coordinate functions.
\begin{corollary} \label{coro:qr_state_prediction}
    Let the assumptions of \theoref{theo:edmd_compression_reduced} hold with $s \geq n$. 
    In addition, assume that $B$ is a coordinate basis for $V$.

    Then, the coordinate matrix $\Pi_{\tilde{E}}$ relative to $\tilde{E}$ is given by
    \begin{align*}
        \Pi_{\tilde{E}} &= \Matrix{Q_{11} & 0} , &
        &\text{with} &
        Q_{11} &= \Matrix{I_n & 0} P^{-1} \Matrix{I_n \\ 0} . 
    \end{align*}
\end{corollary}
\begin{proof}
    We have to express the coordinate functions as linear combinations of the basis elements in $\tilde{E}$. 
    Since $B$ is a coordinate basis, the coordinate functions are given by the first $n$ basis elements in $B$, \ie we have $(\psi_1, \dots, \psi_n) = (\theta_1, \dots, \theta_n)$.
    
    As established in the proof of \lemmaref{lemma:equal_span}, $P^{-1}$ has lower triangular form
    \begin{align*}
        P^{-1} &= \Matrix{
            q_1^1 \\
            q_2^1 & q_2^2 \\
            \vdots & \vdots & \ddots \\
            q_m^1 & q_m^2 & \dots & q_m^m
        } &
        &\implies &
        Q_{11} &= \Matrix{
            q_1^1 \\
            q_2^1 & q_2^2 \\
            \vdots & \vdots & \ddots \\
            q_n^1 & q_n^2 & \dots & q_n^n
        } .
    \end{align*}
    Therefore, $Q_{11}$ relates the basis vectors $(\psi_1, \dots, \psi_n)$ with the basis vectors $(\varphi_1, \dots, \varphi_n)$ via
    \begin{equation*}
        \forall j \in \{1, \dots, n \} \colon \theta_j = \psi_j = \sum_{k=1}^j q_j^k \varphi_k ,
    \end{equation*}
    and the assertion follows.
\end{proof}
\begin{remark}
    Note that the basis $\tilde{E}$ has the advantage that the coordinate matrix $\Pi_{\tilde{E}}$ does not depend on the chosen subspace $S$. 
    This is convenient when setting up an optimisation algorithm that seeks to find an optimal $S$.
    The calculation of $Q_{11}$ can be done easily once the QR decomposition $G_B^\top = \hat{Q} \hat{R}$ is available.
    By definition, we have $P^{-1} = \hat{R}^\top$, and thus $Q_{11} = \hat{R}_{11}^\top$, where $\hat{R}_{11}$ is the upper left block of $\hat{R}$. 
\end{remark}

\section{Finding an Optimal Subspace} \label{sec:optimal_subspace}

In this section, we show that the Grassmann manifold is the natural setting for subspace optimisation, and we introduce the closely-related Stiefel manifold to provide a convenient numerical and algebraic toolkit for our analysis.

Each orthonormal matrix $U \in \rmspace{d}{r}$ identifies an $r$-dimensional subspace of $\mathbb{R}^d$ by the span of its columns.
However, the choice of $U$ for a given subspace is not unique. 
For instance, possible choices to span the two-dimensional subspace $S = \{ (x,y,z) \in \mathbb{R}^3 \mid z = 0 \}$ of $\mathbb{R}^3$ include
\begin{align*}
    U_1 &= \Matrix{1 & 0 \\ 0 & 1 \\ 0 & 0}, &
    U_2 &= \frac{1}{\sqrt{2}} \Matrix{-1 & -1 \\ -1 & 1 \\ 0 & 0}, &
    U_3 &= \frac{1}{\sqrt{2}} \Matrix{1 & 1 \\ -1 & 1 \\ 0 & 0} .
\end{align*}
This is illustrated in \figref{fig:non_unique_basis}. 
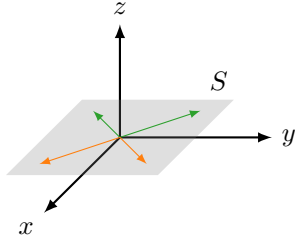
\begin{figure}
   \centering
   \begin{tikzpicture}[>=latex,y={(1cm,0)},x={(-0.5cm,-0.5cm)},z={(0,1cm)}]
        \draw[->,thick] (0,0,0) -- (2,0,0) node[below left] {$x$};
        \draw[->,thick] (0,0,0) -- (0,2,0) node[right] {$y$};
        \draw[->,thick] (0,0,0) -- (0,0,1.5) node[above] {$z$};

        \fill[black!50, nearly transparent] (1,1,0) -- (-1,1,0) -- node[pos=0.1, black, above, opaque] {$S$} (-1,-1,0) -- (1,-1,0) -- cycle;

        \draw[->, mygreen] (0,0,0) -- (-0.7071,-0.7071,0);
        \draw[->, mygreen] (0,0,0) -- (-0.7071,0.7071,0);

        \draw[->, myorange] (0,0,0) -- (0.7071,-0.7071,0);
        \draw[->, myorange] (0,0,0) -- (0.7071,0.7071,0);
   \end{tikzpicture}
   \caption{Subspace $S$, together with the basis given by $U_2$ (green) and $U_3$ (orange).}
   \label{fig:non_unique_basis}
\end{figure}

Formally, the \emph{Stiefel manifold} is defined as the set of all $d \times r$ orthonormal matrices,
\begin{equation*}
	\stiefel{r}{d} = \{ U \in \rmspace{d}{r} \mid U^\top U = I_r \} .
\end{equation*}
Two elements $U_1, U_2 \in \stiefel{r}{d}$ define the same subspace $S$ if and only if there exists an orthogonal matrix $R \in \mathrm{O}(r)$ such that $U_1 = U_2 R$.
It follows that each subspace generated by $U \in \stiefel{r}{d}$ can be uniquely identified with the \emph{equivalence class}
\begin{equation*}
	\eqclass{U} \coloneqq \{ U R \in \stiefel{r}{d} \mid R \in \og{r} \}.
\end{equation*}
Finally, the \emph{Grassmann manifold} $\grass{r}{d}$, see, e.g., \cite{Edelman1998, Bendokat2024, Absil2008, Boumal2023}, is defined as the set of $r$-dimensional subspaces of $\mathbb{R}^d$, and can thus be represented as a \emph{quotient manifold},
\begin{equation*}
	\grass{r}{d} \coloneqq \mathcal{ST}(r,d) / \mathrm{O}(r) = \{ \eqclass{U} \mid U \in \mathcal{ST}(r,d) \} .
\end{equation*}
A \emph{representative} of the equivalence class $[U] \in \grass{r}{d}$ is an element of $[U]$ which may always be written as $UR$ for some $R \in \og{r}$.
To show that a function $\bar{g}$ defined on $\stiefel{r}{d}$ is well-defined on $\grass{r}{d}$, it thus suffices to show that $\bar{g}(U) = \bar{g}(UR)$ for all $U \in \stiefel{r}{d}$ and $R \in \og{r}$.

Recall from \theoref{theo:edmd_compression_reduced} that the matrix representation $K_{\tilde{E}}$ of $C_W \in \End{W}$ by the extended dynamic mode decomposition depends on the chosen basis for $S$ through the matrix $U \in \stiefel{r}{d}$. 
To capture this, we introduce the map
\begin{equation*}
    \hat{K} \colon \begin{cases*}
        \stiefel{r}{d} \to \rmspace{\ell}{\ell} \\
        U \mapsto \hat{K}(U)
    \end{cases*}
\end{equation*}
with
\begin{align*}
    \hat{K}(U) &\coloneqq \bar{U}^\top G_E S_E^\top \bar{U} , & 
    &\text{and} &
    \bar{U} &= \operatorname{blkdiag}(I_s, U) .
\end{align*}

In the following, we rely on $\hat{K}$ to introduce a surrogate for our true objective $\mu_N$ in \eqref{eq:invariance_measure}, where we assume that the prediction horizon $N \in \mathbb{N}$ is fixed.
We do so by introducing an objective function $$g_N \colon \grass{r}{d} \to \mathbb{R}_{\geq 0} $$ on the Grassmann manifold.
The objective should be designed such that a small value for $g_N$ implies a small value for $\mu_N$. 
Our idea is then to solve the problem
\begin{mini}[]
    {\eqclass{U} \in \grass{r}{d}}{g_N(\eqclass{U}) }
    {\label{eq:subspace_edmd_optimisation}}{}
\end{mini}
over the Grassmann manifold.
The objective $g_N$ will introduce a relaxation compared to the true objective $\mu_N$ in \eqref{eq:invariance_measure}. 
Namely, we restrict ourselves to a finite \emph{test data set} $\check{X} \subset X$ of initial states, instead of considering all possible initial states in $X$.
The test data is
\begin{equation*}
    \check{X} = \{ \check{x}_i \}_{i=1}^J \subset X ,
\end{equation*}
drawn i.i.d and collected on top of the training data $D$. 
Moreover, we assume that for each $i \in \{1, \dots, J\}$ and $k \in \{1, \dots, N\}$, the state $x(k,\check{x}_i)$ is available, \ie for each initial state $\check{x}_i \in \check{X}$, we measure the resulting trajectory for the next $N$ timesteps.

Let $U \in \stiefel{r}{d}$ represent a basis for a subspace $S$ of $R$. 
We then introduce the $N$-step prediction error $\bar{g}_N \colon \stiefel{r}{d} \to \mathbb{R}_{\geq 0}$ for $W = T \oplus S$ in state space over the test data as
\begin{equation} \label{eq:residual_state}
    \bar{g}_N(U) \coloneqq \frac{1}{2 J N} \sum_{i=1}^J \sum_{k=1}^N \| x(k, \check{x}_i) - \hat{x}(k, \check{x}_i; S) \|_2^2 .
\end{equation}

We can state the following result.
\begin{proposition}
    Consider $\bar{g}_N \colon \stiefel{r}{d} \to \mathbb{R}_{\geq 0}$ as defined in \eqref{eq:residual_state}. 
    Then, for all $U \in \stiefel{r}{d}$ and $R \in \og{r}$, the equality $\bar{g}_{N}(U) = \bar{g}_{N}(U R)$ holds.
\end{proposition}
\begin{proof}
    This follows directly from \cororef{coro:state_prediction}. 
\end{proof}

We can thus lift the objective to the Grassmannian via
\begin{equation*}
    g_{N} \colon \begin{cases*}
        \grass{r}{d} \to \mathbb{R}_{\geq 0} \\
        \eqclass{U} \mapsto g_{N}(\eqclass{U}) \coloneqq \bar{g}_{N}(U)
    \end{cases*}
\end{equation*}
\begin{remark}
    The chosen objective function is only one of many possible options. Other objectives, tailored to different needs and requirements, can easily be integrated into our framework.
\end{remark}

Let us summarise the procedure derived in the last sections, based on the framework depicted in \figref{fig:sketch_subspaces}.
\begin{enumerate}
    \item 
        Select a coordinate basis $B = (\psi_1, \dots, \psi_M)$ for $V$ such that $T$ is spanned by $(\psi_1, \dots, \psi_s)$, with $s \geq n$. 
    \item 
        Collect the data set $D$ and form the data matrices $G_B$ and $S_B$. 
        Compute $G_B^\top = \hat{Q} \hat{R}$.
    \item
        Introduce the new basis $E = (\varphi_1, \dots, \varphi_M)$ by using the change-of-basis matrix $P = \hat{R}^{-\top}$. 
        Build the transformed data matrices $G_E = P G_B = \hat{Q}^\top$ and $S_E = P S_B$.
    \item 
        Select a reduced dimension $r$ and solve the optimisation problem \eqref{eq:subspace_edmd_optimisation} to obtain $\eqclass{U}$.
    \item
        Consider the resulting subspace $W = T \oplus S \subset V$ with basis $\tilde{E}$, where the lift and the Koopman compression relative to $\tilde{E}$ are
        \begin{align*}
            \tilde{\Phi} &\colon x \mapsto \Matrix{I_s \\ & U^\top} P \Psi(x), &
            &\textnormal{and} &
            K_{\tilde{E}} &= \Matrix{I_s \\ & U^\top} G_E S_E^\top \Matrix{I_s \\ & U} .
        \end{align*}
    \item 
        The state predictions on $W$ are given by $\hat{x}(\cdot,x_0) = \Matrix{Q_{11} & 0} z_{\tilde{E}}(\cdot, x_0)$, where $Q_{11}^\top$ is the upper-left $(n \times n)$-block of $\hat{R}$.
\end{enumerate}

\section{Optimisation over the Grassmann Manifold}
\label{sec:grassmann_optimisation}

For the solution of the optimisation problem \eqref{eq:subspace_edmd_optimisation}, we will apply a trust-region algorithm for Riemannian manifolds as described in \cite{Boumal2023,Absil2008}.
This algorithm requires a Riemannian metric on the Grassman manifold $\grass{r}{d}$, which we determine using its relationship with the Stiefel manifold $\stiefel{r}{d}$.
The equivalence class $[U] \in \grass{r}{d}$ is then optimised by using a representative $U \in \stiefel{r}{d}$, which is only updated in directions parallel\footnote{in the sense of the Riemannian metric} to the tangent space of $[U]$.

\subsection{Riemannian structures of $\stiefel{r}{d}$ and $\grass{r}{d}$}

The Stiefel manifold $\stiefel{r}{d}$ is an embedded submanifold of $\rmspace{d}{r}$, and can be understood as the zero-set of the smooth 
map $\Phi : \rmspace{d}{r} \to \rmspace{r}{r}$ defined by $\Phi(U) := U^\top U - I_r$.
It follows that the tangent space to $\stiefel{r}{d}$ at $U$ can be identified by differentiating this condition, which gives
\begin{align*}
    \tspace{U}{\stiefel{r}{d}} \cong \ker \mathrm{D} \Phi(U) = \{
        V \in \rmspace{d}{r} \mid U^\top V + V^\top U = 0
    \} . 
\end{align*}
The embedding in $\rmspace{d}{r}$ then allows to introduce for each $U \in \stiefel{r}{d}$ the inner product 
\begin{equation*}
    \rmetric{\cdot}{\cdot}{U} \colon \begin{cases*}
        \tspace{U}{\stiefel{r}{d}} \times \tspace{U}{\stiefel{r}{d}} \to \mathbb{R} \\
        (V_1, V_2) \mapsto \rmetric{V_1}{V_2}{U} \coloneqq \operatorname{tr}(V_1^\top V_2)
    \end{cases*}
\end{equation*}
on the tangent space to $\stiefel{r}{d}$ at $U$, which defines a Riemannian metric. 

Recall the definition of the Grassmann manifold as the quotient manifold $\grass{r}{d} = \stiefel{r}{d} / \og{r}$,
formally defined by the map
\begin{equation*}
    \pi \colon \begin{cases*}
        \stiefel{r}{d} \to \grass{r}{d} \\
        U \mapsto \pi(U) \coloneqq \eqclass{U}
    \end{cases*}
\end{equation*}
The derivative of $\pi$ can now be used to split the tangent space at each point into a \emph{vertical space} and \emph{horizontal space}.
The vertical space is defined as the kernel of $\pi$, i.e., 
\[ 
    \ver{U}{\stiefel{r}{d}} = \ker \mathrm{D} \pi(U), 
\]
and consists exactly of the tangent vectors that change the matrix $U$ without changing the equivalence class $[U]$.
The horizontal space is defined as the orthogonal complement to the vertical space $\hor{U}{\stiefel{r}{d}} = (\ver{U}{\stiefel{r}{d}})^\perp$ with respect to the inner product on $\tspace{U}{\stiefel{r}{d}}$ given by the Riemannian metric.
In the particular case of the Stiefel manifold $\stiefel{r}{d}$, the horizontal space at $U$ can be identified with the matrix subspace
\begin{align*}
    \hor{U}{\stiefel{r}{d}} = \{ V \in \rmspace{d}{r} \mid U^\top V = 0 \}.
\end{align*}
The orthogonal projection 
$\Pi_U^\mathrm{hor} \colon \rmspace{d}{r} \to \hor{U}{\stiefel{r}{d}}$ is given by
$
    W \mapsto \Pi_U^\mathrm{hor} W = (I_d - U U^\top) W 
$.

The restriction $\left.\mathrm{D} \pi (U) \right|_{\hor{U}{\stiefel{r}{d}}}$ provides an isomorphism between the horizontal space $\hor{U}{\stiefel{r}{d}}$ and the tangent space $\tspace{\eqclass{U}}{\grass{r}{d}}$.
This identification induces a Riemannian metric on $\grass{r}{d}$ by
\begin{equation*}
    \rmetric{\cdot}{\cdot}{\eqclass{U}} \colon 
    \begin{cases*}
        \tspace{\eqclass{U}}{\grass{r}{d}} \times \tspace{\eqclass{U}}{\grass{r}{d}} \to \mathbb{R} \\
        (\eta_1, \eta_2) \mapsto \rmetric{\eta_1}{\eta_2}{\eqclass{U}} \coloneqq \rmetric{V_1}{V_2}{U}
    \end{cases*}
\end{equation*}
with $V_i = \left.\mathrm{D} \pi (U) \right|_{\hor{U}{\stiefel{r}{d}}}^{-1} \eta_i \in \hor{U}{\stiefel{r}{d}}$ for any $\eqclass{U} \in \grass{r}{d}$, where $U \in [U]$ is an arbitrary representative, the choice of which does not matter.

The consequence of this structure is that applying the trust-region method on the Grassmann manifold is equivalent to applying it on the Stiefel manifold, except that only the horizontal space is considered when computing the gradient and Hessian of the objective at each iteration.
Consider any objective function $f \colon \grass{r}{d} \to \mathbb{R}$ together with $\bar{f} = f \circ \pi \colon \stiefel{r}{d} \to \mathbb{R}$, and let $\bar{\bar{f}} \colon N \to \mathbb{R}$ be a smooth extension for $\bar{f}$ on a neighborhood $N \subseteq \rmspace{d}{r}$ of $\stiefel{r}{d}$.
For any $U \in \eqclass{U} \in \grass{r}{d}$ and $\eta \in \tspace{\eqclass{U}}{\grass{r}{d}}$, the gradient and Hessian of $f$ at $\eqclass{U}$ (and along $\eta$) can be computed by
\begin{align*}
    \gradient f(\eqclass{U}) &= \mathrm{D} \pi (U) \left[ \Pi_U^\mathrm{hor} \gradient \bar{\bar{f}} (U) \right], \\
    \hessian f(\eqclass{U})[\eta] &= \mathrm{D} \pi (U) \left[ 
        \Pi_U^\mathrm{hor} \hessian \bar{\bar{f}}(U)[V] - V U^\top \gradient \bar{\bar{f}} (U)
    \right]
\end{align*}
with $V = \left.\mathrm{D} \pi (U) \right|_{\hor{U}{\stiefel{r}{d}}}^{-1} \eta$ and where the gradient and Hessian of $\bar{\bar{f}}$ can be computed \enquote{as usual}. 
The actual numerical algorithm works directly with the objects in $\hor{U}{\stiefel{r}{d}}$.

\section{Simulation} \label{sec:simulation}

We consider the unforced and undamped Duffing oscillator described by $\dot{x}_1(t) = x_2(t)$ and $\dot{x}_2(t) =  x_1(t) - x_1(t)^3$. 
We consider the dynamics with sampling time $\Delta t = 0.1$. 

We let $V$ be spanned by the monomials up to degree 7, \ie the elements in the basis $B$ are of the form
\begin{equation*}
    \psi_i \colon (x_1, x_2) \mapsto x_1^a x_2^b,
\end{equation*}
where $a,b$ are non-negative integers with $a+b \leq 7$. 
This gives $M = \dim V = 36$.
We select $T \subset V$ to be the two-dimensional subspace spanned by $\psi_1 = \theta_1 \colon x \mapsto x_1$ and $\psi_2 = \theta_2 \colon x \mapsto x_2$, \ie $s=n=2$.
We collect $L = 5000$ training data points $x_i$ on the domain $X = [-1,1] \times [-1,1]$, together with their successor states $y_i = f(x_i)$, which we obtain from a numerical integrator with step-size control. 
From this, we can compute the \enquote{standard} EDMD-based Koopman linear system, where $K_B = (G_B G_B^\top)^{-1} G_B S_B^\top$.

To set up problem \eqref{eq:subspace_edmd_optimisation}, we sample $J = 100$ initial test states $\check{x}_i \in X$ and simulate for $N=20$ time steps. 
Then, using Manopt \cite{Boumal2014}, we solve \eqref{eq:subspace_edmd_optimisation} with $r=3$ to obtain the subspace $S$ and $W = T \oplus S$. 
Recall from \secref{sec:subspace_edmd} that $S$ is spanned by
\begin{equation*}
    (\varphi_1, \varphi_2, u_1^k \alpha_k, u_2^k \alpha_k, u_3^k \alpha_k),
\end{equation*}
where $u_j^k$ denotes the element $(k,j)$ of $U$, $\alpha_k = \varphi_{2+k}$ and the basis $(\varphi_1, \dots, \varphi_{36})$ is obtained by the \enquote{QR transformation} in \eqref{eq:qr_transformation}.

Finally, we compare the state predictions on $V$ and the reduced space $W$. 
For this, we report the mean prediction errors
\begin{align*}
    \varepsilon_\mathrm{r} &\colon x_0 \mapsto \frac{1}{20} \sum_{t=0}^{20} \| x(t, x_0) - \hat{x}(t,x_0; S) \|_2, &
    \varepsilon_\mathrm{f} &\colon x_0 \mapsto \frac{1}{20} \sum_{t=0}^{20} \| x(t, x_0) - \hat{x}(t,x_0) \|_2,
\end{align*}
over a horizon of length $20$, where $\hat{x}(t,x_0; S)$ and $\hat{x}(t,x_0)$ denote the state predictions obtained on $W$ and $V$, respectively. 

First, \figref{fig:duffing_1} shows the absolute difference in the mean errors on the domain $[-1,1] \times [-1,1]$. 
It seems that for most parts, both predictions work equally well; however, close to the two corners $(1,1)$ and $(-1,1)$, the prediction on $W$ starts to work better than the one on $V$. 
This is a first hint that the full Koopman model may have overfitted.
\begin{figure}[htb]
    \centering
    \input{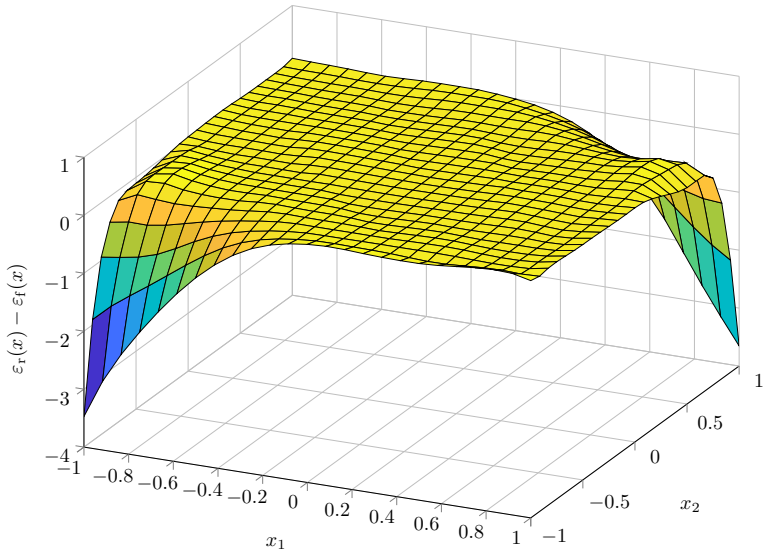}
    \caption{Absolute difference between the mean errors of the state predictions over the horizon $N=20$ based on the reduced and the full dictionary on the domain $[-1,1] \times [-1,1]$.}
    \label{fig:duffing_1}
\end{figure}

As \figref{fig:duffing_2} reveals, this trend is confirmed on the enlarged domain $[-2,2] \times [-2,2]$, where the full model produces significantly higher mean errors. 
This showcases the better generalisation capacities obtained by the reduced model.

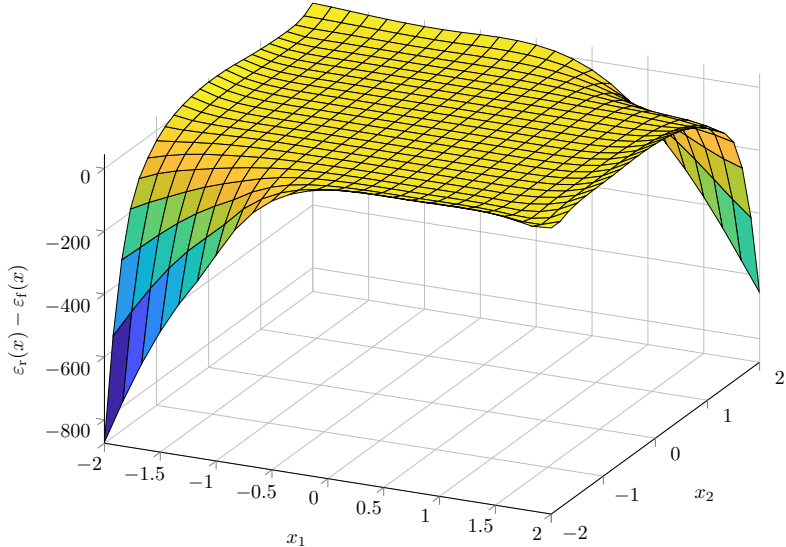
\begin{figure}
    \centering
    \input{duffing_2.tikz}
    \caption{Absolute difference between the mean errors of the state predictions over the horizon $N=20$ based on the reduced and the full dictionary on the domain $[-2,2] \times [-2,2]$.}
    \label{fig:duffing_2}
\end{figure}

\section{Conclusion} \label{sec:conclusion}

We introduce an algorithm operating on the Grassmann manifold to learn an optimal sub-dictionary for building a data-driven surrogate model in the Koopman framework using extended dynamic mode decomposition. 
Our method aims at rendering the subspace, on which we compress the Koopman operator, approximately invariant under its action. 
The used measure for invariance is tailored towards prediction in state space, but other objectives can be easily integrated into the framework.
To design the subspace, we propose a trust-region algorithm to shape the Koopman dictionary of the underlying nonlinear system, leveraging optimisation on manifolds. 
To be more precise, we first divide the tangent spaces into their \textit{horizontal} and \textit{vertical} component. This allows us to rigorously derive a reformulation of the optimisation problem, in which the Stiefel manifold is replaced by the Grassmannian. 
Numerical tractability is maintained by relying on Stiefel representatives. 
Finally, we demonstrate the effectiveness of the proposed method using a numerical example. 

There are several open questions for future work, for example, a more detailed quantization w.r.t. the data requirements.
Furthermore, variants of the presented cost function are conceivable, depending on the intended use of the surrogate model or incorporating system knowledge such as translational or rotational invariances in the learning process. 
In addition, one may enforce the sparsity using the ideas of~\cite{pan2021sparsity}.


\begin{appendices}

\section{Compression under a change-of-basis} \label{app:compression_change_of_basis}

First, we illustrate the basis dependence with a brief example.
\begin{example}
    Consider the basis $E = (\varphi_1, \dots, \varphi_M) \coloneqq (2 \psi_1, \dots, 2 \psi_M)$ for $V$, \ie $P = 2 I_M$. 
    For each $j \in \{1,\dots,M\}$, we observe that
    \begin{align*}
        C_V \varphi_j = 2 C_V \psi_j = 2 a_j^k \psi_k = 2 a_j^k (\frac{1}{2} \varphi_k) = a_j^k \varphi_k = C_V \psi_j
    \end{align*}
    holds by linearity of $C_V$, so that the matrix $K_E$ of $C_V$ relative to $E$ is given by $K_E = K_B$. 
    The lift $\Phi \colon \mathbb{R}^n \to \mathbb{R}^m$ relative to $E$ satisfies $\Phi = 2 \Psi$ for this example.
    This leads to the Koopman linear system
    \begin{align*}
        z(t+1) &= K_E^\top z(t), &
        z(0) &= \Phi(x_0),
    \end{align*}
    relative to $E$. 
    Hence, its trajectories satisfy 
    \begin{equation*}
        \forall x_0 \in X \colon z_E(\cdot, x_0) = 2 z_B(\cdot, x_0) ,
    \end{equation*}
    which showcases the variance of the lifted state to the chosen basis. 
    This is natural, since $z_E(\cdot,x_0)$ approximates $\Phi(x(\cdot,x_0))$ and not $\Psi(x(\cdot,x_0))$. %
\end{example}
We now move on to prove \lemmaref{lemma:koopman_predicitions_basis}. 
The proof is supported by \figref{fig:change_of_basis}.

\begin{figure}
    \centering
    \begin{tikzpicture}[node distance=2cm]
        \node at (0,0) (V1) [] {$V$};
        \node (V2) [right of=V1] {$V$};
        \node (V3) [right of=V2] {$V$};
        \node (V4) [right of=V3] {$V$};
        \node (Rm1) [below of=V1] {$\mathbb{R}^M$};
        \node (Rm2) [below of=V2] {$\mathbb{R}^M$};
        \node (Rm3) [below of=V3] {$\mathbb{R}^M$};
        \node (Rm4) [below of=V4] {$\mathbb{R}^M$};

        \draw[->] (V1) -- node[above] {\footnotesize $\mathantt{id}_V$} (V2);
        \draw[->] (V2) -- node[above] {\footnotesize $C_V$} (V3);
        \draw[->] (V3) -- node[above] {\footnotesize $\mathantt{id}_V$} (V4);

        \draw[->] (Rm1) -- node[above] {\footnotesize $P^\top$} (Rm2);
        \draw[->] (Rm2) -- node[above] {\footnotesize $K_B$} (Rm3);
        \draw[->] (Rm3) -- node[above] {\footnotesize $P^{-\top}$} (Rm4);
        \draw[->] (Rm1) to [out=-30, in=210] node[above] {\footnotesize $K_E$} (Rm4);

        \draw[->] (Rm1) -- node[left] {\footnotesize $T_E$} (V1);
        \draw[->] (Rm2) -- node[left] {\footnotesize $T_B$} (V2);
        \draw[->] (Rm3) -- node[left] {\footnotesize $T_B$} (V3);
        \draw[->] (Rm4) -- node[left] {\footnotesize $T_E$} (V4);
    \end{tikzpicture}
    \caption{Commutative diagram, showcasing the change-of-basis. With $T_B$ we denote the isomorphism of vector spaces that sends any $(v^1, \dots, v^m) \in \mathbb{R}^m$ to $v^j \psi_j \in V$, and analogously for $T_E$.}
    \label{fig:change_of_basis}
\end{figure}
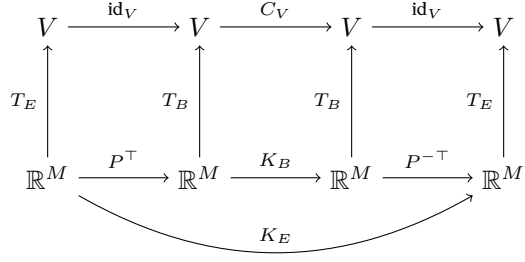

\begin{proof}[Proof of \lemmaref{lemma:koopman_predicitions_basis}]
    Fix any $x_0 \in X$, and introduce for notational convenience $w(\cdot,x_0) \coloneqq P z_B(\cdot,x_0)$. 
    We thus have to show $z_E(\cdot, x_0) = w(\cdot, x_0)$.
    
    We do this by showing that (i) $z_B(\cdot,x_0)$ and $w(\cdot,x_0)$ satisfy the same difference equation, and (ii) that they have the same initial value. 

    By construction, $w(\cdot, x_0)$ is the trajectory resulting from the transformed system
    \begin{align*}
        w(t+1) &= P K_B^\top P^{-1} w(t), &
        w(0) &= P \Psi(x_0) .
    \end{align*}
    Hence, we wish to relate $K_E$ to $K_B$ and $P$. 
    Denoting
    \begin{equation*}
        P^{-1} = \Matrix{
            q_1^1 & \dots & q_1^M \\
            \vdots & \ddots & \vdots \\
            q_M^1 & \dots & q_M^M
        }
    \end{equation*}
    allows us to express the basis vectors in $B$ as linear combinations of the basis vectors in $E$.
    To see this, let $\Psi$ and $\Phi$ denote the lifts relative to $B$ and $E$, respectively. 
    By definition of $P$, we have
    \begin{equation} \label{eq:transformed_lifts}
        \forall x \in X \colon \Phi(x) = P \Psi(x) .
    \end{equation}
    This implies for each $j \in \{1, \dots, M \}$ the equality
    \begin{equation*}
        \psi_j = q_j^k \varphi_k .
    \end{equation*} 
    Recall that $K_B$ is described by \eqref{eq:compression_dictionary_basis}.
    Then, the matrix representation $K_E$ of $C_V$ relative to $E$ is given by
    \begin{equation*}
        C_V \varphi_j = p_j^i C_V \psi_i = p_j^i a_i^s \psi_s = p_j^i a_i^s q_s^k \varphi_k \eqqcolon \tilde{a}_j^k \varphi_k .
    \end{equation*}
    In matrix form, this is expressed as
    \begin{align*}
        K_E^\top &= \Matrix{
            \tilde{a}_1^1 & \dots &  \tilde{a}_1^M \\
            \vdots & \ddots & \vdots \\
            \tilde{a}_M^1 & \dots &  \tilde{a}_M^M
        } = \Matrix{
            p_1^i a_i^s q_s^1 & \dots & p_1^i a_i^s q_s^M \\
            \vdots & \ddots & \vdots \\
            p_M^i a_i^s q_s^1 & \dots & p_M^i a_i^s q_s^M
        } \\
        &= \Matrix{
            p_1^1 & \dots & p_1^M \\
            \vdots & \ddots & \vdots \\
            p_M^1 & \dots & p_M^M
        } \Matrix{
            a_1^1 & \dots & a_1^M \\
            \vdots & \ddots & \vdots \\
            a_M^1 & \dots & a_M^M
        } \Matrix{
            q_1^1 & \dots & q_1^M \\
            \vdots & \ddots & \vdots \\
            q_M^1 & \dots & q_M^M
        } \\
        &= P K_B^\top P^{-1}
    \end{align*}
    Hence, $P K_B^\top P^{-1} = K_E^\top$, which implies that $w(\cdot,x_0)$ and $z_E(\cdot,x_0)$ satisfy the same difference equation. 
    This shows (i).

    Additionally, \eqref{eq:transformed_lifts} implies
    \begin{equation*}
        w(0,x_0) = P z_B(0,x_0) = P \Psi(x_0) = \Phi(x_0) = z_E(0,x_0) ,
    \end{equation*}
    showing (ii).
                
    Hence, $z_E(\cdot, x_0) = w(\cdot,x_0) = P z_B(\cdot,x_0)$ as claimed.
\end{proof}

\section{Coordinate Projection}\label{app:coordinate_projection}

\begin{proposition} \label{prop:left_inverse_lift}
    Let $B = (\psi_1, \dots, \psi_M)$ be an arbitrary basis for a coordinate subspace $V$.
    Then, the lift $\Psi \colon \mathbb{R}^n \to \operatorname{im} \Psi$ relative to $B$ has an inverse $\Psi^{-1} \colon \operatorname{im} \Psi \to \mathbb{R}^n$. 
\end{proposition}
\begin{proof}
    The existence of the inverse can be shown with the coordinate matrix $\Pi_B$ relative to $B$. 
    By construction of $\Pi_B$, we have 
    \begin{equation*}
        \forall x \in \mathbb{R}^n \colon \Pi_B \Psi(x) = (c_1^j \psi_j(x), \dots, c_n^j \psi_j(x))
        = (\theta_1(x), \dots, \theta_n(x)) 
        = x ,
    \end{equation*}
    where the $c_i^j$ are the components of the coordinate functions relative to $B$. 
    Hence, the map $$\Psi^{-1} \coloneqq \left. \Pi_B \right|_{\operatorname{im} \Psi} \colon \operatorname{im} \Psi \to \mathbb{R}^n$$ is a left inverse for $\Psi$. 

    Additionally, from the above computation we have for each $z = \Psi(x) \in  \operatorname{im} \Psi$ that
    \begin{equation*}
        (\Psi \circ \Psi^{-1})(z) = \Psi(\Pi_B \Psi(x)) = \Psi(x) = z,
    \end{equation*}
    which establishes the claim.
\end{proof}

\begin{corollary}
    Let $B = (\psi_1, \dots, \psi_M)$ be an arbitrary basis for a coordinate subspace $V$.
    
    Then, the coordinate projection $\pi_\Psi$ relative to $B$ satisfies $\left. \pi_\Psi \right|_{\operatorname{im} \Psi} = \mathantt{id}_{\operatorname{im} \Psi}$. 
\end{corollary}
\begin{proof}
    This follows directly from \propref{prop:left_inverse_lift}.
    For each $z \in \operatorname{im}(\Psi)$, we have by definition 
    \begin{equation*}
        \pi_\Psi(z) = \Psi(\Pi_B z) = (\Psi \circ \Psi^{-1})(z) = z .
    \end{equation*}
\end{proof}

\begin{proof}[Proof of \cororef{coro:state_prediction}]
    Let $P \in \gl{m}$ be the change-of-basis matrix from $B$ to $E$.
    Then, \lemmaref{lemma:koopman_predicitions_basis} gives $z_E(\cdot, x_0) = P z_B(\cdot, x_0)$. 
    Additionally, we want to express the coordinate matrix $\Pi_E$ relative to $E$ in terms of $\Pi_B$ and $P$.
    Since $V$ is a coordinate subspace and $E$ is a basis for $V$, we have for each $j \in \{1, \dots, n\}$ the existence of real numbers $d_j^k$ such that
    \begin{equation*}
        \theta_j = d_j^k \varphi_k = d_j^k p_k^i \psi_i = c_j^i \psi_i
    \end{equation*}
    holds, with $c_j^i = d_j^k p_k^i$. 
    Putting this in matrix form, we obtain
    \begin{equation*}
        \Pi_B = \Matrix{
            c_1^1 & \dots & c_1^m \\
            \vdots & \ddots & \vdots \\
            c_n^1 & \dots & c_n^m
        } = \Matrix{
            d_1^1 & \dots & d_1^m \\
            \vdots & \ddots & \vdots \\
            d_n^1 & \dots & d_n^m
        } \Matrix{
            p_1^1 & \dots & p_1^m \\
            \vdots & \ddots & \vdots \\
            p_m^1 & \dots & p_m^m
        } = \Pi_E P .
    \end{equation*}
    This is equivalent to $\Pi_E = \Pi_B P^{-1}$, which finally implies
    \begin{equation*}
        \Pi_E z_E(\cdot, x_0) = \Pi_E P z_B(\cdot,x_0) = \Pi_B P^{-1} P z_B(\cdot,x_0) 
        = \Pi_B z_B(\cdot,x_0) ,
    \end{equation*}
    as claimed.
\end{proof}

\section{Orthonormal Basis} \label{app:ONB}

\begin{lemma} \label{lemma:edmd_stiefel_basis}
    Let $R$ be a finite-dimensional subspace of $H$, and suppose that $(\alpha_1, \dots, \alpha_d)$ is a basis for $R$.
    
    Then, for each subspace $S \subset R$ with $0 < \dim S = r < \dim R$, there is a basis $(\beta_1, \dots, \beta_r)$ for $S$ such that the matrix $U$ of components relative to $(\alpha_{1}, \dots, \alpha_d)$ has orthonormal columns.
\end{lemma}
\begin{proof}
    Since $(\alpha_{1}, \dots, \alpha_d)$ is a basis for $R$, the subspace $R$ can be identified with $\mathbb{R}^d$ via the isomorphism of vector spaces given by
    \begin{equation*}
        \Phi \colon \begin{cases*}
            \mathbb{R}^d \to R \\
            u = (u^1, \dots, u^d) \mapsto u^k \alpha_k
        \end{cases*}
    \end{equation*}
    Now, fix any $r$-dimensional subspace $S$ of $R$.
    Let $\hat{S}$ be the $r$-dimensional subspace of $\mathbb{R}^d$ that satisfies $\ran \left. \Phi \right|_{\hat{S}} = S$.
    Equip $\hat{S}$ with an inner product by restricting the standard inner product 
    \begin{equation*}
        \forall u,v \in \mathbb{R}^d \colon \iprod{u}{v}_{\mathbb{R}^d} \coloneqq u^\top v
    \end{equation*}
    in $\mathbb{R}^d$ to $\hat{S}$. 
    Select an orthonormal basis $(u_1, \dots, u_r)$ for $\hat{S}$ relative to this inner product (\eg by relying on the Gram-Schmidt process), and denote for each $j \in \{1, \dots, r \}$ and $k \in \{1, \dots, d\}$ the $k$-th component of $u_j$ by $u_j^k$. 
    Then, $(\beta_1, \dots, \beta_r) \coloneqq (\Phi u_1, \dots, \Phi u_r)$ is a basis for $S$, and we have
    \begin{equation*}
        \forall j \in \{1, \dots, r \} \colon \beta_j = \Phi u_j = u_j^k \alpha_{k}
    \end{equation*}
    by definition of $\Phi$. 
    Thus, the matrix
    \begin{equation*}
        U = \Matrix{u_1^1 & \dots & u_r^1 \\ \vdots & \ddots & \vdots \\ u_1^d & \dots & u_r^d} 
        = \Matrix{u_1 & \dots & u_r}
    \end{equation*}
    of components relative to $(\alpha_1, \dots, \alpha_d)$ satisfies
    \begin{equation*}
        U^\top U = \Matrix{u_1^\top \\ \vdots \\ u_r^\top} \Matrix{u_1 & \dots & u_r}
        = \Matrix{
            u_1^\top u_1 & \dots & u_1^\top u_r \\
            \vdots & \ddots & \vdots \\
            u_r^\top u_1 & \dots & u_r^\top u_r
        } 
        = \Matrix{
            \iprod{u_1}{u_1}_{\mathbb{R}^d} & \dots & \iprod{u_1}{u_r}_{\mathbb{R}^d} \\
            \vdots & \ddots & \vdots \\
            \iprod{u_r}{u_1}_{\mathbb{R}^d} & \dots & \iprod{u_r}{u_r}_{\mathbb{R}^d}
        } 
        = I_r, 
    \end{equation*}
    since $(u_1, \dots, u_r)$ forms an orthonormal basis for $\hat{S}$. 
    This completes the proof.
\end{proof}

\section{Computing a Compression} \label{app:computing_compression}

For this work to be self-contained, we derive a general procedure for obtaining a compression of $K$ on a finite-dimensional subspace $V$ of $H$ in our notation. 
Here, we also use the results from \appref{app:projections}.

Let us refine the notation a bit compared to the main text. 
If $P_V$ is a projection of $H$ with $\ran P = V$ and $\ker P = Z$, then we call $C_V = P_V \left. K \right|_V$ the compression of $K$ on $V$ along $Z$.
As a first step, we show how a compression can be derived from a non-degenerate bilinear form. 
Then, we introduce a way to define this bilinear form from data, which finally leads to the EDMD compression from \secref{sec:subspace_edmd}
Throughout, we fix a finite-dimensional subspace $V$ of $H$, and let $B = (\psi_1, \dots, \psi_M)$ be basis for $V$. 

\subsection*{Based on a Bilinear Form}
To compute a compression of $K$ in the first place, we must have access to some information on how the Koopman operator acts on vectors in $V$. 
Here, we assume that this information can be expressed in terms of a \emph{bilinear form}
\begin{equation*}
    b \colon \begin{cases*}
        H \times H \to \mathbb{R} \\
        (\psi, \varphi) \mapsto b(\psi, \varphi)
    \end{cases*}
\end{equation*}
The key idea in working with this bilinear form is that we can evaluate for each $j,k \in \{1, \dots, M\}$ the expression $b(K \psi_j, \psi_k)$, which will allow us to compute the matrix representation of the compression.
In the actual algorithm that we use and which we present in \secref{sec:subspace_edmd}, the bilinear form will be defined from data, but we present here the approach in full generality here.  

However, not any bilinear form suffices; we have one critical requirement on the bilinear form. 
We say that $b$ is \emph{non-degenerate on $V$} if the linear map
\begin{equation*}
    b^\prime \colon \begin{cases*}
        V \to V^\ast \\
        \psi \mapsto b(\cdot, \psi)
    \end{cases*}
\end{equation*}
is one-to-one. 
\begin{corollary} \label{coro:bilinear_iso}
    Let $V \subset H$ be a finite-dimensional subspace of $H$, and suppose that $b \colon H \times H \to \mathbb{R}$ is a bilinear form on $H$.
    
    Then, if $b$ is non-degenerate on $V$, $b^\prime$ is a vector space isomorphism.
\end{corollary}
\begin{proof}
    By assumption, $b^\prime$ is linear and one-to-one, so it is left to show that $b^\prime$ is onto. 

    Linearity and injectivity of $b^\prime$ imply $\dim \ker b^\prime = 0$, and then the rank-nullity theorem implies $\dim \ran b^\prime = \dim V = \dim V^\ast$. 
    Hence, $\ran b^\prime = V^\ast$, and the assertion follows.
\end{proof}
Now, define the matrices
\begin{align} \label{eq:matrices_bilinear}
    H_B &= \Matrix{
        b(\psi_1,\psi_1) & \dots & b(\psi_M,\psi_1) \\
        \vdots & \ddots & \vdots \\
        b(\psi_1,\psi_M) & \dots & b(\psi_M,\psi_M)  
    } , & 
    A_B &= \Matrix{
        b(K \psi_1,\psi_1) & \dots & b(K \psi_M,\psi_1) \\
        \vdots & \ddots & \vdots \\
        b(K \psi_1,\psi_M) & \dots & b(K \psi_M,\psi_M)  
    } ,
\end{align}
relative to the basis $B$.
Here, it becomes clear that we must be able to evaluate $b(K \psi_j, \psi_k) \in \mathbb{R}$, in order to build $A_B$. 
Note that $K \psi_j \in H$ itself must not be available.
If that is possible, we have the following result.
\begin{theorem} \label{theo:koopman_compression}
    Let $V \subset H$ be a finite-dimensional subspace of $H$, and suppose that $B = (\psi_1, \dots, \psi_M)$ is a basis for $V$. 
    Further, assume that a bilinear form $b \colon H \times H \to \mathbb{R}$ is available, which is non-degenerate on $V$, and let $H_B, A_B$ be the matrices relative to $B$ defined in \eqref{eq:matrices_bilinear}.
    For each $j \in \{1, \dots, M\}$, introduce $\bar{\sigma}^j \coloneqq b(\cdot, \psi_j) \colon H \to \mathbb{R}$, and consider the subspace $Z = \ker \bar{\sigma}^1 \cap \dots \cap \ker \bar{\sigma}^M$ of $H$.

    Then, the matrix representation $K_B$ of the compression $C_V$ of $K$ on $V$ along $Z$ is 
    \begin{equation*}
        K_B = H_B^{-1} A_B .
    \end{equation*}
\end{theorem}
\begin{proof}
    As an initial remark, note that in the notation of the theorem, we can equivalently express
    \begin{align*}
        H_B &= \Matrix{
            \bar{\sigma}^1(\psi_1) & \dots & \bar{\sigma}^1(\psi_M) \\
            \vdots & \ddots & \vdots \\
            \bar{\sigma}^M(\psi_1) & \dots & \bar{\sigma}^M(\psi_M) 
        } , & 
        A_B &= \Matrix{
            \bar{\sigma}^1(K \psi_1) & \dots & \bar{\sigma}^1(K \psi_M) \\
            \vdots & \ddots & \vdots \\
            \bar{\sigma}^M(K \psi_1) & \dots & \bar{\sigma}^M(K \psi_M)
        } .
    \end{align*}
    
    For the proof, we proceed step-by-step. 
    First, we show that $H_B$ has full rank.
    Then, we establish that $H$ is the direct sum of $V$ and $Z$, such that $C_V$ exists, \cf \theoref{theo:projection_direct_sum}. 
    Finally, we can calculate the matrix representation by using the projection $P_V$ on $V$ along $Z$.

    For the first step, let $B^\ast = (\nu^1, \dots, \nu^M)$ denote the unique dual basis of $B$.
    Under the made assumptions, \cororef{coro:bilinear_iso} implies that
    \begin{equation*}
        (\sigma^1, \dots, \sigma^M) \coloneqq (\left. \bar{\sigma}^1 \right|_V, \dots, \left. \bar{\sigma}^M \right|_V) = (b^\prime(\psi_1), \dots, b^\prime(\psi_M))
    \end{equation*}
    is a basis for $V^\ast$. 
    This implies the existence of real numbers $y_k^j$ such that 
    \begin{equation*}
        \forall j \in \{1, \dots, M\} \colon \nu^j = y_k^j \sigma^k 
    \end{equation*}
    holds. 
    Collecting these components in matrix from yields
    \begin{equation*}
        Y \coloneqq \Matrix{
            y_1^1 & \dots & y_M^1 \\
            \vdots & \ddots & \vdots \\
            y_1^M & \dots & y_M^M
        } \in \gl{M} ,
    \end{equation*}
    which has full rank, because $Y$ is the change-of-basis matrix from $(\sigma^1, \dots, \sigma^M)$ to $B^\ast$.
    Because $B^\ast$ is the dual basis of $B$, we additionally have
    \begin{equation*}
        I_M = \Matrix{
            \nu^1(\psi_1) & \dots & \nu^1(\psi_M) \\
            \vdots & \ddots & \vdots \\
            \nu^M(\psi_1) & \dots & \nu^M(\psi_M) \\
        }
        = \Matrix{
            y_1^1 & \dots & y_M^1 \\
            \vdots & \ddots & \vdots \\
            y_1^M & \dots & y_1^M
        } \Matrix{
            \bar{\sigma}^1(\psi_1) & \dots & \bar{\sigma}^1(\psi_M) \\
            \vdots & \ddots & \vdots \\
            \bar{\sigma}^M(\psi_1) & \dots & \bar{\sigma}^M(\psi_M)
        } = Y H_B .
    \end{equation*}
    This implies $H_B = Y^{-1} \in \gl{M}$, as required, from which we can also conclude $Y = H_B^{-1}$.
    
    For the second step, we use the previous result to extend the dual basis $B^\ast$ to all of $H$ by defining for each $j \in \{1, \dots, M\}$ the linear functional
    \begin{equation*}
        \bar{\nu}^j \coloneqq y_k^j \bar{\sigma}^k \colon H \to \mathbb{R} .
    \end{equation*}
    We now wish to show that $Z = \ker \bar{\nu}^1 \cap \dots \cap \ker \bar{\nu}^M$ is true. 
    Then, $H = V \oplus Z$ follows from \theoref{theo:direct_sum_finite_dual}. 
    To see this equality, note that for each $\psi \in H$ the equivalences
    \begin{align*}
        &\forall j \in \{1, \dots, M\} \colon 0 = \bar{\nu}^j(\psi) = y_k^j \bar{\sigma}^k(\psi) &
        &\iff &
        &\Matrix{
            0 \\ \vdots \\ 0
        } = \Matrix{y_1^1 & \dots & y_M^1 \\
            \vdots & \ddots & \vdots \\
            y_1^M & \dots & y_M^M
        } \Matrix{
            \bar{\sigma}^1(\psi) \\ \vdots \\ \bar{\sigma}^m(\psi)
        } \\
        &&&\iff &
        &\forall j \in \{1, \dots, M\} \colon 0 = \bar{\sigma}^j(\psi) ,
    \end{align*}
    hold, where we used in the second step that $Y$ has full rank. 
    Hence, we may conclude $\ker \bar{\nu}^1 \cap \dots \cap \ker \bar{\nu}^M = Z$ and thus $H = V \oplus Z$ follows from \theoref{theo:direct_sum_finite_dual}. 
    
    Finally, recall $C_V = P_V \left. K \right|_V$, where $P_V$ is the projection on $V$ along $Z$.    
    To construct $K_B$, we must compute for each $j \in \{1, \dots, M\}$ the components $a_j^k$ of
    \begin{equation*}
        C_V \psi_j = a_j^k \psi_k \in V
    \end{equation*}
    relative to $B$. 
    The proof of \theoref{theo:direct_sum_finite_dual} implies that they are given by 
    \begin{equation*}
        a_j^k = P_V K \psi_j = \bar{\nu}^k(K \psi_j) = y_i^k \bar{\sigma}^i(K \psi_j) .
    \end{equation*}
    Then we have by definition of $K_B$ and our previous result $Y = H_B^{-1}$ that
    \begin{equation*}
        K_B = \Matrix{
            a_1^1 & \dots & a_M^1 \\
            \vdots & \ddots & \vdots \\
            a_1^M & \dots & a_M^M
        } = 
        \Matrix{
            y_1^1 & \dots & y_M^1 \\
            \vdots & \ddots & \vdots \\
            y_1^M & \dots & y_M^M
        } \Matrix{
            \bar{\sigma}^1(K \psi_1) & \dots & \bar{\sigma}^1(K \psi_M) \\
            \vdots & \ddots & \vdots \\
            \bar{\sigma}^M(K \psi_1) & \dots & \bar{\sigma}^M(K \psi_M)
        } = H_B^{-1} A_B ,
    \end{equation*}
    as claimed.
\end{proof}
Now, fix a subspace $W$ of $V$ of dimension $\ell$.
The above result allows us to express a Koopman compression on $W \subset V$ in terms of a Koopman compression on $V$. 
We represent $W$ by selecting a basis $\tilde{B} = (\omega_1, \dots, \omega_\ell)$ for $W$. 
Because $B$ is a basis for $V$ and $W \subseteq V$, we can express the basis vectors of $W$ as a linear combination of the basis vectors of $V$. 
Specifically, we assume
\begin{equation*}
    \forall j \in \{1, \dots, \ell \} \colon \omega_j = \bar{u}_j^k \psi_k ,
\end{equation*}
and collect the components in the matrix
\begin{equation} \label{eq:basis_reduced}
    \bar{U} = \Matrix{
        \bar{u}_1^1 & \dots & \bar{u}_\ell^1 \\
        \vdots & \ddots & \vdots \\
        \bar{u}_1^M & \dots & \bar{u}_\ell^M
    } \in \rmspace{M}{\ell} ,
\end{equation}
which has full column rank.
\begin{corollary} \label{coro:koopman_compression_sub}
    Let the assumptions of \theoref{theo:koopman_compression} hold. 
    In addition, assume that $W$ is a subspace of $V$, and that $\tilde{B} = (\omega_1, \dots, \omega_\ell)$ is a basis for $W$. 
    Suppose that the relation between $\tilde{B}$ and $B$ is described by the matrix $\bar{U}$ in \eqref{eq:basis_reduced}. 
    For each $j \in \{1, \dots, \ell\}$, introduce $\bar{\tau}^j \coloneqq b(\cdot, \omega_j) \colon H \to \mathbb{R}$, and consider the subspace $F = \ker \bar{\tau}^1 \cap \dots \cap \ker \bar{\tau}^m$.

    Then, the matrix representation $K_{\tilde{B}}$ of the compression $C_W$ of $K$ on $W$ along $F$ is 
    \begin{equation*}
        K_{\tilde{B}} = \left( \bar{U}^\top H_B \bar{U} \right)^{-1} \bar{U}^\top A_B \bar{U} .
    \end{equation*}
\end{corollary}
\begin{proof}
    First, note that the assumption that $b$ is non-degenerate on $V$ together with $W \subseteq V$ implies that $b$ is non-degenerate on $W$.
    Following this, the key part of the proof is to establish that the matrices
    \begin{align*}
        H_{\tilde{B}} &= \Matrix{
            b(\omega_1,\omega_1) & \dots & b(\omega_\ell,\omega_1) \\
            \vdots & \ddots & \vdots \\
            b(\omega_1,\omega_\ell) & \dots & b(\omega_\ell,\omega_\ell)  
        } , & 
        A_{\tilde{B}} &= \Matrix{
            b(K \omega_1,\omega_1) & \dots & b(K \omega_\ell,\omega_1) \\
            \vdots & \ddots & \vdots \\
            b(K \omega_1,\omega_\ell) & \dots & b(K \omega_\ell,\omega_\ell)  
        } ,
    \end{align*}
    satisfy $H_{\tilde{B}} = \bar{U}^\top H_B \bar{U}$ and $A_{\tilde{B}} = \bar{U}^\top A_B \bar{U}$. 
    Then, the assertion follows from \theoref{theo:koopman_compression}.

    Note that for each $j,k \in \{1, \dots, \ell\}$ by definition of $\tilde{B}$ we have
    \begin{equation*}
        b(\omega_j, \omega_k) = \bar{u}_j^i \bar{u}_k^a b(\psi_i, \psi_a),
    \end{equation*}
    and by linearity of $K$ we also have
    \begin{equation*}
        b(K \omega_j, \omega_k) = \bar{u}_k^a b(\bar{u}_j^i K \psi_i, \psi_a) = \bar{u}_j^i \bar{u}_k^a b(K \psi_i, \psi_a)
    \end{equation*}
    Putting this in matrix form yields the claimed relation.
\end{proof}
This result enables us to directly construct from the data $H_B, A_B$ -- which defines a Koopman linear system on $V$ -- a Koopman linear system on a subspace $W \subset V$ and evaluate its invariance. 

\subsection*{Based on Data}

Now, assume that the training data $D$ in \eqref{eq:edmd_training_data} has been collected. 
By using this data, we are equipped to define the bilinear form 
\begin{equation} \label{eq:bilinear_form_data}
    b \colon \begin{cases*}
        H \times H \to \mathbb{R} \\
        (\psi,\varphi) \mapsto \sum_{i=1}^L \psi(x_i) \cdot \varphi(x_i)
    \end{cases*}
\end{equation}
Note that the training data $D$ enables us to evaluate for each $j,k \in \{1, \dots, M\}$ the term $b(K \psi_j, \psi_k)$, since
\begin{equation} \label{eq:bilinear_form_data_koopman}
    b(K \psi_j, \psi_k) = \sum_{i=1}^L (K \psi_j)(x_i) \cdot \psi_k(x_i) = \sum_{i=1}^L (\psi_j \circ f)(x_i) \cdot \psi_k(x_i) = \sum_{i=1}^L \psi_j(y_i) \cdot \psi_k(x_i) 
\end{equation}
holds by definition of the Koopman operator. 
This is illustrated in \figref{fig:edmd_training_data}.

\begin{figure}
    \centering
    \begin{tikzpicture}[>=latex, x=2.75cm, y=1.25cm]
        \begin{scope}[xshift=-3.5cm]
            \draw [] (-1,-1) rectangle ++(2,2);
            \node at(0.95,0.85) (X) {\footnotesize $X$};
            
            \node at (-0.9,-0.8)  (x1) [circle, draw, inner sep=1pt, fill=white] {};
            \node at (-0.1,-0.4)  (y1) [circle, draw, inner sep=1pt, fill=white] {};
            \draw[->, dashed] (x1) to[out=45, in=165] node[below] {\footnotesize $f$} (y1);
            \node at ([xshift=0.75em]x1) [] {\footnotesize $x_1$};
            \node at ([yshift=-0.6em]y1.south) [] {\footnotesize $y_1$};
            
            \node at (0.1,0.3)  (xi) [circle, draw, inner sep=1pt, fill=white] {};
            \node at (0.7,-0.3)  (yi) [circle, draw, inner sep=1pt, fill=white] {};
            \draw[->, dashed] (xi) to[out=0,in=180] node[above, pos=0.7] {\footnotesize $f$} (yi);
            \node at ([xshift=-0.75em]xi) [] {\footnotesize $x_i$};
            \node at ([yshift=-0.6em]yi.south) [] {\footnotesize $y_i$};
            
            \node at (-0.8,0.4)  (xM) [circle, draw, inner sep=1pt, fill=white] {};
            \node at (-0.3,0.8)  (yM) [circle, draw, inner sep=1pt, fill=white] {};
            \draw[->, dashed] (xM) to[out=0,in=180] node[above] {\footnotesize $f$} (yM);
            \node at ([xshift=-0.75em]xM) [] {\footnotesize $x_L$};
            \node at ([yshift=-0.6em]yM.south) [] {\footnotesize $y_L$};
            
        \end{scope}
        \begin{scope}[xshift=3.5cm]
            \draw[->] (-1,0) -- ++(2,0) node [below] {\footnotesize $\mathbb{R}$};
            
            \draw (0.7,2pt) -- node[name=psiy] {} (0.7,-2pt);
            \draw (-0.6,2pt) -- node[name=Kpsix] {} (-0.6,-2pt);
        \end{scope}
        \draw[->, myblue, dash dot dot] (yi.east) to[bend right] node[pos=0.5, above] {\footnotesize $\psi_j$} (psiy);
        \draw[->, mygreen, dash dot dot] (xi.east) to[bend left] node[pos=0.5, below] {\footnotesize $K \psi_j$} (psiy);
        \draw[->, myred, dash dot dot] (xi.east) to[out=0,in=145] node[pos=0.6, below] {\footnotesize $C_{V} \psi_j$} (Kpsix);
    \end{tikzpicture}
    \caption{Training data for the EDMD algorithm}
    \label{fig:edmd_training_data}
\end{figure}

An important object related to the bilinear form in \theoref{theo:koopman_compression} for expressing a Koopman compression were the data matrices \eqref{eq:matrices_bilinear} relative to a basis $B$. 
In this case, they are related to the data matrices $G_B, S_B$ defined in \eqref{eq:edmd_matrices}.
For convenience, we state the following direct consequence of the definition of $b$ in \eqref{eq:bilinear_form_data}.
\begin{corollary} \label{coro:matrices_bilinear_data}
    Let $V$ be a finite-dimensional subspace of $H$, and suppose that $B = (\psi_1, \dots, \psi_m)$ is a basis for $V$. 
    Let $G_B, S_B$ denote the data matrices in \eqref{eq:edmd_matrices} relative to $B$.
    
    Then, the matrices $H_B$ and $A_B$ defined in \eqref{eq:matrices_bilinear} relative the bilinear form in \eqref{eq:bilinear_form_data} are given by $H_B = G_B G_B^\top$ and $A_B = G_B S_B^\top$.
\end{corollary}
\begin{proof}
    This follows directly from the definitions of $H_B, A_B$ and by considering \eqref{eq:bilinear_form_data_koopman}.
\end{proof}
To build a Koopman linear system with these ingredients, we have to ensure that the bilinear form
$b$ derived from the training data is non-degenerate on $V$. 
To this end, we must make an additional assumption on the training data $D$ and the basis $B$, which we formalise below.
\begin{lemma} \label{lemma:non_degenerate_data}
    Let the assumptions of \cororef{coro:matrices_bilinear_data} hold. 
    Then, the following statements are equivalent.
    \begin{enumerate}
        \item[(i)] The data matrix $G_B$ has full row rank $m$.
        \item[(ii)] The bilinear form $b$ defined in \eqref{eq:bilinear_form_data} is non-degenerate on $V$. 
    \end{enumerate}
\end{lemma}
\begin{proof}
    As $(ii)$ is a property of the subspace $V$ and not of a specific basis, we first assure ourselves that $(i)$ implies that the data matrix has full row rank in \emph{any} basis. 
    Suppose that $E = (\varphi_1, \dots, \varphi_M)$ is another basis for $V$, and that $P \in \gl{M}$ is the change-of-basis matrix from $B$ to $E$, \cf \eqref{eq:change_of_basis_matrix}. 
    Then, we have
    \begin{align*}
        {G}_E &= \Matrix{
            \varphi_1({x}_1) & \dots & \varphi_1({x}_L) \\
            \vdots & \ddots & \vdots \\
            \varphi_M({x}_1) & \dots & \varphi_M({x}_L)
        } = \Matrix{
            p_1^k \psi_k({x}_1) & \dots & p_1^k \psi_k({x}_L) \\
            \vdots & \ddots & \vdots \\
            p_M^k \psi_k({x}_1) & \dots & p_M^k \psi_k({x}_L)
        } \\
        &= \Matrix{
            p_1^1 & \dots & p_1^M \\
            \vdots & \ddots & \vdots \\
            p_M^1 & \dots & p_M^M
        } \Matrix{
            \psi_1({x}_1) & \dots & \psi_1({x}_L) \\
            \vdots & \ddots & \vdots \\
            \psi_M({x}_1) & \dots & \psi_M({x}_L)
        } = P G_B .
    \end{align*}
    Since $P$ has full rank, $G_E$ has full row rank if and only if $G_B$ has full row rank, as desired.

    We now prove the claimed equivalence.
    Recall that $b$ is non-degenerate on $V$ if $$b^\prime \colon V \ni \psi \mapsto b(\cdot,\psi) \in V^\ast$$ is one-to-one. 
    Since $b^\prime$ is linear, this is equivalent to $\ker b^\prime = \{ 0 \}$.

    Assume that $\psi = q^k \psi_k \in \ker b^\prime$.
    This means that $\nu \coloneqq b^\prime(\psi) \in V^\ast$ maps \emph{any} vector in $V$ to zero, which in particular implies
    \begin{equation*}
        \forall j \in \{1, \dots, M\} \colon 0 = \nu(\psi_j) = b(\psi_j, \psi) = q^k b(\psi_j, \psi_k) .
    \end{equation*}
    Denoting $q = (q^1, \dots, q^M) \in \mathbb{R}^M$ and collecting the above conditions yields
    \begin{align} \label{eq:kernel_bilinear_data}
        \Matrix{0 \\ \vdots \\ 0} = \Matrix{
            q^k b(\psi_1, \psi_k) \\ \vdots \\ q^k b(\psi_M, \psi_k)
        } &= \Matrix{
            b(\psi_1, \psi_1) & \dots & b(\psi_1, \psi_M) \\
            \vdots & \ddots & \vdots \\
            b(\psi_M, \psi_1) & \dots & b(\psi_M, \psi_M)
        } \Matrix{q^1 \\ \vdots \\ q^M} &
        &\iff &
        H_B q &= 0 .
    \end{align}
    First, assume that $b$ is non-degenerate on $V$. 
    This means that the only solution to \eqref{eq:kernel_bilinear_data} is $q = 0$, which implies $H_B \in \gl{M}$.
    Using \cororef{coro:matrices_bilinear_data}, we have $H_B = G_B G_B^\top$.
    Since the rank of $G_B$ and ${G}_B {G}_B^\top$ coincide, this implies that $G_B$ has full row rank.

    The reverse direction uses a similar argument. 
    Assume that $G_B$ has full row rank. 
    Using \cororef{coro:matrices_bilinear_data} again, this means that $H_B = G_B G_B^\top$ is invertible, so that the only solution to \eqref{eq:kernel_bilinear_data} is $q=0$. 
    Hence, $\psi \in \ker b^\prime$ must be the zero vector, which implies that $b$ is non-degenerate on $V$.
\end{proof}

Hence, assuming that $G_B$ defined in \eqref{eq:edmd_matrices} does have full row rank, we can directly apply \theoref{theo:koopman_compression}, where the bilinear form $b$ is given by \eqref{eq:bilinear_form_data}, to obtain a subspace $Z \subset H$ and the compression $C_V$ of $K$ on $V$ along $Z$, which is represented relative to $B$ by the 
\begin{equation*}
    K_B = (G_B G_B^\top)^{-1} G_B S_B^\top .
\end{equation*}
We refer to this specific compression of $K$ as the \emph{EDMD compression on $V$}, so that the subspace $Z$ that we project along is always implicitly understood to be the one defined in \theoref{theo:koopman_compression}.

\begin{remark}
    The EDMD compression $C_V$ on $V$ has the alternative interpretation of being a minimiser to the regression problem
    \begin{argmini*}
        {A \in \End{V}}{\sum_{i=1}^L \sum_{j=1}^m (\psi_j(y_i) - (A \psi_j)(x_i))^2}
        {}{}
    \end{argmini*}
    It can be shown that $C_V$ converges almost surely to the compression $C_V^\star$ obtained by the $L^2$-orthogonal projection on $V$ for $L \to \infty$ by using the law of large numbers \cite{Korda2018}. 

    This means that for each $\zeta \in V$ we have for the $L^2$-orthogonal projection the orthogonality property
    \begin{equation*}
        \forall \psi \in V \colon \iprod{K \zeta - C_V^\star \zeta}{\psi} = 0, 
    \end{equation*}
    while in our formulation we have the weaker property 
    \begin{equation*}
        \forall \psi \in V \colon b(K \zeta - C_V \zeta, \psi) = 0,
    \end{equation*}
    which means that we do not project orthogonally with respect to the inner product on $H$, but with respect to the \enquote{orthogonality} notion induced by the data-enabled bilinear form $b$.
\end{remark}
Subsequently, we may use \cororef{coro:koopman_compression_sub} to represent the EDMD compression $C_W$ on a subspace $W \subseteq V$.

\section{Projections}\label{app:projections}

For this work to be self-contained, we include some basic facts about projections on vector spaces. 

Let $V$ be a vector space over the field $\mathbb{F}$.
Suppose that $W$ and $S$ are subspaces of $V$. 
We define their \emph{sum} $W+S$ as 
\begin{equation*}
    W + S = \{ 
        w + s \mid w \in W \land s \in S
    \} \subseteq V,
\end{equation*}
which is, again, a subspace of $V$.

We say that $V$ is the \emph{direct sum} of $W$ and $S$ if for every vector $v$ in $V$ there exist \emph{unique} vectors $w \in W$ and $s \in S$ such that $v = s+w$.
Then, we write $V = W \oplus S$.

\begin{theorem} \label{theo:direct_sum}
    Let $V$ be a vector space, and let $W, S$ be subspaces of $V$.

     Then, $V$ is the direct sum of $W$ and $S$ if and only if $W + S = V$ and $W \cap S = \{ 0 \}$.
\end{theorem}
\begin{proof}
    First, suppose that $W + S = V$ and $W \cap S = \{ 0 \}$.
    Let $v \in V$ be arbitrary. 
    The assumption $V = W + S$ implies the existence of $w \in W$ and $s \in S$ such that $v = w + s$. 
    To establish that $V$ is the direct sum of $W$ and $S$, we have to prove that these elements $w, s$ are uniquely determined. 
    
    Assume that $\bar{w} \in W$ and $\bar{s} \in S$ are chosen such that $v = \bar{w} + \bar{s}$. 
    Then, we have
    \begin{align*}
        w + s &= \bar{w} + \bar{s} &
        &\iff &
        w - \bar{w} = \bar{s} - s
    \end{align*}
    By linearity of both $W$ and $S$, we have $w - \bar{w} \in W$ and $s - \bar{s} \in S$. 
    By assumption, $W \cap S = \{ 0 \}$ holds, from which we conclude
    \begin{equation*}
        w - \bar{w} = 0 \land \bar{s} - s = 0 .
    \end{equation*}
    Hence, $w = \bar{w}$ and $s = \bar{s}$, as claimed.

    Conversely, suppose that $V$ is the direct sum of $W$ and $S$. 
    We first establish $V = W + S$
    Let $v \in V$ be arbitrary. 
    By assumption, there exist unique vectors $w$ in $W$ and $s$ in $S$ such that $v = w + s$. 
    This shows $V \subseteq W + S$.  
    From linearity of $V$ and because $W$ and $S$ are subspaces of $V$, we have $W + S \subseteq V$. 
    Combining both gives $V = W + S$. 

    Finally, we prove $W \cap S = \{0\}$. 
    Suppose that $v$ lies in $W \cap S$. 
    Then, we have $v = v + 0$, with $v \in W$ and $0 \in S$. 
    At the same time, we have $v = 0 + v$, with $0 \in W$ and $v \in S$. 
    Because the expression for $v$ as the sum of a vector in $W$ and a vector in $S$ is unique, we conclude $v=0$.
\end{proof}
\begin{theorem} \label{theo:direct_sum_finite_dual}
    Let $V$ be a vector space, and let $W$ be a finite-dimensional subspace of $V$. 
    Suppose that $(w_1, \dots, w_m)$ is a basis for $W$, with corresponding dual basis $(\varphi^1, \dots, \varphi^m)$. 
    Assume that for each $j \in \{1, \dots, m\}$, there exists a linear functional $\bar{\varphi}^j \colon V \to \mathbb{R}$ that is an extension of $\varphi^j$ to all of $V$. 

    Then, $V$ is the direct sum of $W$ and $S = \ker \bar{\varphi}^1 \cap \dots \cap \ker \bar{\varphi}^m$.
\end{theorem}
\begin{proof}
    To prove the theorem, we rely on \theoref{theo:direct_sum}.

    We first establish $V = W + S$. 
    Since $W$ and $S$ are both subspaces of $V$, we surely have $W+S \subseteq V$.
    For the reverse direction, let $v \in V$ be arbitrary. 
    For each $j \in \{1, \dots, m\}$, set $p^j = \bar{\varphi}^j(v) \in \mathbb{R}$ and define $w \coloneqq p^j w_j \in W$. 
    Then, we have for each $j \in \{1, \dots, m\}$ that
    \begin{align*}
        \bar{\varphi}^j(v - w) &= \bar{\varphi}^j(v) - \bar{\varphi}^j(w)
        = \bar{\varphi}^j(v) - \varphi^j(w) 
        = \bar{\varphi}^j(v) - p^k \varphi^j(w_k)\
        = \bar{\varphi}^j(v) - p^k \delta_k^j \\
        &= \bar{\varphi}^j(v) - p^j
        = p^j - p^j 
        = 0 
    \end{align*}
    holds. 
    Thus, we have $v - w \in S$, and then $v = w + (v-w)$ implies $V \subseteq W + S$, so that we can conclude $V = W + S$.

    Next, we show $W \cap S = \{ 0 \}$. 
    Suppose that $v$ lies in $W \cap S$. 
    From $v \in W$ follows the existence of real numbers $p^j$ such that $v = p^j w_j$. 
    Additionally, since $v \in S$, we have for each $j \in \{1, \dots, m \}$ that $\bar{\varphi}^j(v) = 0$. 
    Combining this yields 
    \begin{equation*}
        \forall j \in \{ 1, \dots, m\} \colon 0 = \bar{\varphi}^j(v) = p^k \bar{\varphi}^j(w_k) 
        = p^k \varphi^j(w_k) 
        = p^k \delta_j^k 
        = p^j,
    \end{equation*}
    and hence $v = p^j w_j$ is the zero vector, as claimed.

    Finally, the assertion follows from \theoref{theo:direct_sum}.
\end{proof}

A \emph{projection} of $V$ is a linear map $P$ on $V$ such that $P \circ P = P$.

\begin{theorem} \label{theo:projection_direct_sum}
    Let $V$ be a vector space, and let $W$ and $S$ be subspaces of $V$ such that $V$ is the direct sum of $W$ and $S$. 
    Define the linear map $P$ on $V$ that sends each vector $v = w + s$ in $V$ to its unique component $w$ in $W$ relative to $S$.
    
    Then, $P$ is a projection of $V$ with $\ran P = W$ and $\ker P = S$, called the projection on $W$ along $S$.
\end{theorem}
\begin{proof}
    Let $v \in V$ be arbitrary. 
    Let $w \in W$ and $s \in S$ be the unique vectors that satisfy $v = w + s$. 
    Then, we have $P v = w \in W$ by definition. 
    The assumption $V = W \oplus S$ implies that $w \in W$ has only a trivial component in $S$, and therefore
    \begin{equation*}
        (P \circ P) v = P w = P(w + 0) = w = P v.
    \end{equation*}
    This shows that $P$ is a projection of $V$. 
    The range and null space of $P$ follow directly from its definition.
\end{proof} 





\end{appendices}


\bibliography{sn-bibliography}

\end{document}

%% file: duffing_2.tikz
%
%
\definecolor{mycolor1}{rgb}{0.85098,0.85098,0.85098}%
\definecolor{mycolor2}{rgb}{0.07059,0.07059,0.07059}%
\begin{tikzpicture}[scale=0.75]

\begin{axis}[%
width=4.521in,
height=3.566in,
at={(0.758in,0.481in)},
scale only axis,
xmin=-2,
xmax=2,
tick align=outside,
ymin=-2,
ymax=2,
zmin=-875,
zmax=45,
xlabel=$x_1$,
ylabel=$x_2$,
zlabel=$\varepsilon_\mathrm{r}(x) - \varepsilon_\mathrm{f}(x)$,
axis background/.style={fill=none},
axis x line*=bottom,
axis y line*=left,
axis z line*=left,
xmajorgrids,
ymajorgrids,
zmajorgrids
]

\addplot3[%
surf,
shader=flat corner, draw=black, z buffer=sort, colormap={mymap}{[1pt] rgb(0pt)=(0.2422,0.1504,0.6603); rgb(1pt)=(0.2444,0.1534,0.6728); rgb(2pt)=(0.2464,0.1569,0.6847); rgb(3pt)=(0.2484,0.1607,0.6961); rgb(4pt)=(0.2503,0.1648,0.7071); rgb(5pt)=(0.2522,0.1689,0.7179); rgb(6pt)=(0.254,0.1732,0.7286); rgb(7pt)=(0.2558,0.1773,0.7393); rgb(8pt)=(0.2576,0.1814,0.7501); rgb(9pt)=(0.2594,0.1854,0.761); rgb(11pt)=(0.2628,0.1932,0.7828); rgb(12pt)=(0.2645,0.1972,0.7937); rgb(13pt)=(0.2661,0.2011,0.8043); rgb(14pt)=(0.2676,0.2052,0.8148); rgb(15pt)=(0.2691,0.2094,0.8249); rgb(16pt)=(0.2704,0.2138,0.8346); rgb(17pt)=(0.2717,0.2184,0.8439); rgb(18pt)=(0.2729,0.2231,0.8528); rgb(19pt)=(0.274,0.228,0.8612); rgb(20pt)=(0.2749,0.233,0.8692); rgb(21pt)=(0.2758,0.2382,0.8767); rgb(22pt)=(0.2766,0.2435,0.884); rgb(23pt)=(0.2774,0.2489,0.8908); rgb(24pt)=(0.2781,0.2543,0.8973); rgb(25pt)=(0.2788,0.2598,0.9035); rgb(26pt)=(0.2794,0.2653,0.9094); rgb(27pt)=(0.2798,0.2708,0.915); rgb(28pt)=(0.2802,0.2764,0.9204); rgb(29pt)=(0.2806,0.2819,0.9255); rgb(30pt)=(0.2809,0.2875,0.9305); rgb(31pt)=(0.2811,0.293,0.9352); rgb(32pt)=(0.2813,0.2985,0.9397); rgb(33pt)=(0.2814,0.304,0.9441); rgb(34pt)=(0.2814,0.3095,0.9483); rgb(35pt)=(0.2813,0.315,0.9524); rgb(36pt)=(0.2811,0.3204,0.9563); rgb(37pt)=(0.2809,0.3259,0.96); rgb(38pt)=(0.2807,0.3313,0.9636); rgb(39pt)=(0.2803,0.3367,0.967); rgb(40pt)=(0.2798,0.3421,0.9702); rgb(41pt)=(0.2791,0.3475,0.9733); rgb(42pt)=(0.2784,0.3529,0.9763); rgb(43pt)=(0.2776,0.3583,0.9791); rgb(44pt)=(0.2766,0.3638,0.9817); rgb(45pt)=(0.2754,0.3693,0.984); rgb(46pt)=(0.2741,0.3748,0.9862); rgb(47pt)=(0.2726,0.3804,0.9881); rgb(48pt)=(0.271,0.386,0.9898); rgb(49pt)=(0.2691,0.3916,0.9912); rgb(50pt)=(0.267,0.3973,0.9924); rgb(51pt)=(0.2647,0.403,0.9935); rgb(52pt)=(0.2621,0.4088,0.9946); rgb(53pt)=(0.2591,0.4145,0.9955); rgb(54pt)=(0.2556,0.4203,0.9965); rgb(55pt)=(0.2517,0.4261,0.9974); rgb(56pt)=(0.2473,0.4319,0.9983); rgb(57pt)=(0.2424,0.4378,0.9991); rgb(58pt)=(0.2369,0.4437,0.9996); rgb(59pt)=(0.2311,0.4497,0.9995); rgb(60pt)=(0.225,0.4559,0.9985); rgb(61pt)=(0.2189,0.462,0.9968); rgb(62pt)=(0.2128,0.4682,0.9948); rgb(63pt)=(0.2066,0.4743,0.9926); rgb(64pt)=(0.2006,0.4803,0.9906); rgb(65pt)=(0.195,0.4861,0.9887); rgb(66pt)=(0.1903,0.4919,0.9867); rgb(67pt)=(0.1869,0.4975,0.9844); rgb(68pt)=(0.1847,0.503,0.9819); rgb(69pt)=(0.1831,0.5084,0.9793); rgb(70pt)=(0.1818,0.5138,0.9766); rgb(71pt)=(0.1806,0.5191,0.9738); rgb(72pt)=(0.1795,0.5244,0.9709); rgb(73pt)=(0.1785,0.5296,0.9677); rgb(74pt)=(0.1778,0.5349,0.9641); rgb(75pt)=(0.1773,0.5401,0.9602); rgb(76pt)=(0.1768,0.5452,0.956); rgb(77pt)=(0.1764,0.5504,0.9516); rgb(78pt)=(0.1755,0.5554,0.9473); rgb(79pt)=(0.174,0.5605,0.9432); rgb(80pt)=(0.1716,0.5655,0.9393); rgb(81pt)=(0.1686,0.5705,0.9357); rgb(82pt)=(0.1649,0.5755,0.9323); rgb(83pt)=(0.161,0.5805,0.9289); rgb(84pt)=(0.1573,0.5854,0.9254); rgb(85pt)=(0.154,0.5902,0.9218); rgb(86pt)=(0.1513,0.595,0.9182); rgb(87pt)=(0.1492,0.5997,0.9147); rgb(88pt)=(0.1475,0.6043,0.9113); rgb(89pt)=(0.1461,0.6089,0.908); rgb(90pt)=(0.1446,0.6135,0.905); rgb(91pt)=(0.1429,0.618,0.9022); rgb(92pt)=(0.1408,0.6226,0.8998); rgb(93pt)=(0.1383,0.6272,0.8975); rgb(94pt)=(0.1354,0.6317,0.8953); rgb(95pt)=(0.1321,0.6363,0.8932); rgb(96pt)=(0.1288,0.6408,0.891); rgb(97pt)=(0.1253,0.6453,0.8887); rgb(98pt)=(0.1219,0.6497,0.8862); rgb(99pt)=(0.1185,0.6541,0.8834); rgb(100pt)=(0.1152,0.6584,0.8804); rgb(101pt)=(0.1119,0.6627,0.877); rgb(102pt)=(0.1085,0.6669,0.8734); rgb(103pt)=(0.1048,0.671,0.8695); rgb(104pt)=(0.1009,0.675,0.8653); rgb(105pt)=(0.0964,0.6789,0.8609); rgb(106pt)=(0.0914,0.6828,0.8562); rgb(107pt)=(0.0855,0.6865,0.8513); rgb(108pt)=(0.0789,0.6902,0.8462); rgb(109pt)=(0.0713,0.6938,0.8409); rgb(110pt)=(0.0628,0.6972,0.8355); rgb(111pt)=(0.0535,0.7006,0.8299); rgb(112pt)=(0.0433,0.7039,0.8242); rgb(113pt)=(0.0328,0.7071,0.8183); rgb(114pt)=(0.0234,0.7103,0.8124); rgb(115pt)=(0.0155,0.7133,0.8064); rgb(116pt)=(0.0091,0.7163,0.8003); rgb(117pt)=(0.0046,0.7192,0.7941); rgb(118pt)=(0.0019,0.722,0.7878); rgb(119pt)=(0.0009,0.7248,0.7815); rgb(120pt)=(0.0018,0.7275,0.7752); rgb(121pt)=(0.0046,0.7301,0.7688); rgb(122pt)=(0.0094,0.7327,0.7623); rgb(123pt)=(0.0162,0.7352,0.7558); rgb(124pt)=(0.0253,0.7376,0.7492); rgb(125pt)=(0.0369,0.74,0.7426); rgb(126pt)=(0.0504,0.7423,0.7359); rgb(127pt)=(0.0638,0.7446,0.7292); rgb(128pt)=(0.077,0.7468,0.7224); rgb(129pt)=(0.0899,0.7489,0.7156); rgb(130pt)=(0.1023,0.751,0.7088); rgb(131pt)=(0.1141,0.7531,0.7019); rgb(132pt)=(0.1252,0.7552,0.695); rgb(133pt)=(0.1354,0.7572,0.6881); rgb(134pt)=(0.1448,0.7593,0.6812); rgb(135pt)=(0.1532,0.7614,0.6741); rgb(136pt)=(0.1609,0.7635,0.6671); rgb(137pt)=(0.1678,0.7656,0.6599); rgb(138pt)=(0.1741,0.7678,0.6527); rgb(139pt)=(0.1799,0.7699,0.6454); rgb(140pt)=(0.1853,0.7721,0.6379); rgb(141pt)=(0.1905,0.7743,0.6303); rgb(142pt)=(0.1954,0.7765,0.6225); rgb(143pt)=(0.2003,0.7787,0.6146); rgb(144pt)=(0.2061,0.7808,0.6065); rgb(145pt)=(0.2118,0.7828,0.5983); rgb(146pt)=(0.2178,0.7849,0.5899); rgb(147pt)=(0.2244,0.7869,0.5813); rgb(148pt)=(0.2318,0.7887,0.5725); rgb(149pt)=(0.2401,0.7905,0.5636); rgb(150pt)=(0.2491,0.7922,0.5546); rgb(151pt)=(0.2589,0.7937,0.5454); rgb(152pt)=(0.2695,0.7951,0.536); rgb(153pt)=(0.2809,0.7964,0.5266); rgb(154pt)=(0.2929,0.7975,0.517); rgb(155pt)=(0.3052,0.7985,0.5074); rgb(156pt)=(0.3176,0.7994,0.4975); rgb(157pt)=(0.3301,0.8002,0.4876); rgb(158pt)=(0.3424,0.8009,0.4774); rgb(159pt)=(0.3548,0.8016,0.4669); rgb(160pt)=(0.3671,0.8021,0.4563); rgb(161pt)=(0.3795,0.8026,0.4454); rgb(162pt)=(0.3921,0.8029,0.4344); rgb(163pt)=(0.405,0.8031,0.4233); rgb(164pt)=(0.4184,0.803,0.4122); rgb(165pt)=(0.4322,0.8028,0.4013); rgb(166pt)=(0.4463,0.8024,0.3904); rgb(167pt)=(0.4608,0.8018,0.3797); rgb(168pt)=(0.4753,0.8011,0.3691); rgb(169pt)=(0.4899,0.8002,0.3586); rgb(170pt)=(0.5044,0.7993,0.348); rgb(171pt)=(0.5187,0.7982,0.3374); rgb(172pt)=(0.5329,0.797,0.3267); rgb(173pt)=(0.547,0.7957,0.3159); rgb(175pt)=(0.5748,0.7929,0.2941); rgb(176pt)=(0.5886,0.7913,0.2833); rgb(177pt)=(0.6024,0.7896,0.2726); rgb(178pt)=(0.6161,0.7878,0.2622); rgb(179pt)=(0.6297,0.7859,0.2521); rgb(180pt)=(0.6433,0.7839,0.2423); rgb(181pt)=(0.6567,0.7818,0.2329); rgb(182pt)=(0.6701,0.7796,0.2239); rgb(183pt)=(0.6833,0.7773,0.2155); rgb(184pt)=(0.6963,0.775,0.2075); rgb(185pt)=(0.7091,0.7727,0.1998); rgb(186pt)=(0.7218,0.7703,0.1924); rgb(187pt)=(0.7344,0.7679,0.1852); rgb(188pt)=(0.7468,0.7654,0.1782); rgb(189pt)=(0.759,0.7629,0.1717); rgb(190pt)=(0.771,0.7604,0.1658); rgb(191pt)=(0.7829,0.7579,0.1608); rgb(192pt)=(0.7945,0.7554,0.157); rgb(193pt)=(0.806,0.7529,0.1546); rgb(194pt)=(0.8172,0.7505,0.1535); rgb(195pt)=(0.8281,0.7481,0.1536); rgb(196pt)=(0.8389,0.7457,0.1546); rgb(197pt)=(0.8495,0.7435,0.1564); rgb(198pt)=(0.86,0.7413,0.1587); rgb(199pt)=(0.8703,0.7392,0.1615); rgb(200pt)=(0.8804,0.7372,0.165); rgb(201pt)=(0.8903,0.7353,0.1695); rgb(202pt)=(0.9,0.7336,0.1749); rgb(203pt)=(0.9093,0.7321,0.1815); rgb(204pt)=(0.9184,0.7308,0.189); rgb(205pt)=(0.9272,0.7298,0.1973); rgb(206pt)=(0.9357,0.729,0.2061); rgb(207pt)=(0.944,0.7285,0.2151); rgb(208pt)=(0.9523,0.7284,0.2237); rgb(209pt)=(0.9606,0.7285,0.2312); rgb(210pt)=(0.9689,0.7292,0.2373); rgb(211pt)=(0.977,0.7304,0.2418); rgb(212pt)=(0.9842,0.733,0.2446); rgb(213pt)=(0.99,0.7365,0.2429); rgb(214pt)=(0.9946,0.7407,0.2394); rgb(215pt)=(0.9966,0.7458,0.2351); rgb(216pt)=(0.9971,0.7513,0.2309); rgb(217pt)=(0.9972,0.7569,0.2267); rgb(218pt)=(0.9971,0.7626,0.2224); rgb(219pt)=(0.9969,0.7683,0.2181); rgb(220pt)=(0.9966,0.774,0.2138); rgb(221pt)=(0.9962,0.7798,0.2095); rgb(222pt)=(0.9957,0.7856,0.2053); rgb(223pt)=(0.9949,0.7915,0.2012); rgb(224pt)=(0.9938,0.7974,0.1974); rgb(225pt)=(0.9923,0.8034,0.1939); rgb(226pt)=(0.9906,0.8095,0.1906); rgb(227pt)=(0.9885,0.8156,0.1875); rgb(228pt)=(0.9861,0.8218,0.1846); rgb(229pt)=(0.9835,0.828,0.1817); rgb(230pt)=(0.9807,0.8342,0.1787); rgb(231pt)=(0.9778,0.8404,0.1757); rgb(232pt)=(0.9748,0.8467,0.1726); rgb(233pt)=(0.972,0.8529,0.1695); rgb(234pt)=(0.9694,0.8591,0.1665); rgb(235pt)=(0.9671,0.8654,0.1636); rgb(236pt)=(0.9651,0.8716,0.1608); rgb(237pt)=(0.9634,0.8778,0.1582); rgb(238pt)=(0.9619,0.884,0.1557); rgb(239pt)=(0.9608,0.8902,0.1532); rgb(240pt)=(0.9601,0.8963,0.1507); rgb(241pt)=(0.9596,0.9023,0.148); rgb(242pt)=(0.9595,0.9084,0.145); rgb(243pt)=(0.9597,0.9143,0.1418); rgb(244pt)=(0.9601,0.9203,0.1382); rgb(245pt)=(0.9608,0.9262,0.1344); rgb(246pt)=(0.9618,0.932,0.1304); rgb(247pt)=(0.9629,0.9379,0.1261); rgb(248pt)=(0.9642,0.9437,0.1216); rgb(249pt)=(0.9657,0.9494,0.1168); rgb(250pt)=(0.9674,0.9552,0.1116); rgb(251pt)=(0.9692,0.9609,0.1061); rgb(252pt)=(0.9711,0.9667,0.1001); rgb(253pt)=(0.973,0.9724,0.0938); rgb(254pt)=(0.9749,0.9782,0.0872); rgb(255pt)=(0.9769,0.9839,0.0805)}, mesh/rows=25]
table[row sep=crcr, point meta=\thisrow{c}] {%
x	y	z	c\\
-2	-2	-871.354904907937	-871.354904907937\\
-2	-1.83333333333333	-553.23748442835	-553.23748442835\\
-2	-1.66666666666667	-331.191571112605	-331.191571112605\\
-2	-1.5	-189.539088436418	-189.539088436418\\
-2	-1.33333333333333	-112.59261717351	-112.59261717351\\
-2	-1.16666666666667	-67.2681763873973	-67.2681763873973\\
-2	-1	-32.3943393631042	-32.3943393631042\\
-2	-0.833333333333333	-7.72852661497448	-7.72852661497448\\
-2	-0.666666666666667	8.63060706202899	8.63060706202899\\
-2	-0.5	19.5023954597055	19.5023954597055\\
-2	-0.333333333333333	26.7142523984847	26.7142523984847\\
-2	-0.166666666666667	31.139491209708	31.139491209708\\
-2	0	33.2005653681172	33.2005653681172\\
-2	0.166666666666667	33.2473802037662	33.2473802037662\\
-2	0.333333333333333	31.6273912020628	31.6273912020628\\
-2	0.5	28.4714438310656	28.4714438310656\\
-2	0.666666666666667	24.3221423424715	24.3221423424715\\
-2	0.833333333333333	20.3415664038728	20.3415664038728\\
-2	1	17.0509462663757	17.0509462663757\\
-2	1.16666666666667	14.581373785249	14.581373785249\\
-2	1.33333333333333	13.0172822183654	13.0172822183654\\
-2	1.5	12.582086591055	12.582086591055\\
-2	1.66666666666667	13.6374559148907	13.6374559148907\\
-2	1.83333333333333	22.6301161052406	22.6301161052406\\
-2	2	43.1928112434754	43.1928112434754\\
-1.83333333333333	-2	-730.578863455009	-730.578863455009\\
-1.83333333333333	-1.83333333333333	-465.419413108456	-465.419413108456\\
-1.83333333333333	-1.66666666666667	-281.782051536119	-281.782051536119\\
-1.83333333333333	-1.5	-159.066186287117	-159.066186287117\\
-1.83333333333333	-1.33333333333333	-87.5436974615727	-87.5436974615727\\
-1.83333333333333	-1.16666666666667	-50.3983250884515	-50.3983250884515\\
-1.83333333333333	-1	-25.9369611678011	-25.9369611678011\\
-1.83333333333333	-0.833333333333333	-9.65290429875527	-9.65290429875527\\
-1.83333333333333	-0.666666666666667	0.489214320466004	0.489214320466004\\
-1.83333333333333	-0.5	6.97928932827443	6.97928932827443\\
-1.83333333333333	-0.333333333333333	11.2431223153561	11.2431223153561\\
-1.83333333333333	-0.166666666666667	13.8934724458549	13.8934724458549\\
-1.83333333333333	0	15.1760841164084	15.1760841164084\\
-1.83333333333333	0.166666666666667	15.240608780563	15.240608780563\\
-1.83333333333333	0.333333333333333	14.2120860481619	14.2120860481619\\
-1.83333333333333	0.5	12.2092176660836	12.2092176660836\\
-1.83333333333333	0.666666666666667	9.83996095951149	9.83996095951149\\
-1.83333333333333	0.833333333333333	7.74014748098541	7.74014748098541\\
-1.83333333333333	1	6.09909999923518	6.09909999923518\\
-1.83333333333333	1.16666666666667	4.97524300848379	4.97524300848379\\
-1.83333333333333	1.33333333333333	4.51942059297351	4.51942059297351\\
-1.83333333333333	1.5	5.02040983487632	5.02040983487632\\
-1.83333333333333	1.66666666666667	10.0724495963549	10.0724495963549\\
-1.83333333333333	1.83333333333333	21.5051711985099	21.5051711985099\\
-1.83333333333333	2	39.9048726017318	39.9048726017318\\
-1.66666666666667	-2	-602.863770156769	-602.863770156769\\
-1.66666666666667	-1.83333333333333	-382.78946168802	-382.78946168802\\
-1.66666666666667	-1.66666666666667	-232.28285847372	-232.28285847372\\
-1.66666666666667	-1.5	-132.467219030052	-132.467219030052\\
-1.66666666666667	-1.33333333333333	-69.6821671566688	-69.6821671566688\\
-1.66666666666667	-1.16666666666667	-36.7980431599991	-36.7980431599991\\
-1.66666666666667	-1	-19.0338298400846	-19.0338298400846\\
-1.66666666666667	-0.833333333333333	-8.2164961185794	-8.2164961185794\\
-1.66666666666667	-0.666666666666667	-2.21133275072017	-2.21133275072017\\
-1.66666666666667	-0.5	1.37719255639723	1.37719255639723\\
-1.66666666666667	-0.333333333333333	3.74891912817138	3.74891912817138\\
-1.66666666666667	-0.166666666666667	5.29131057594828	5.29131057594828\\
-1.66666666666667	0	6.0830333302545	6.0830333302545\\
-1.66666666666667	0.166666666666667	6.15719470942746	6.15719470942746\\
-1.66666666666667	0.333333333333333	5.54571349484246	5.54571349484246\\
-1.66666666666667	0.5	4.3836354515431	4.3836354515431\\
-1.66666666666667	0.666666666666667	3.18513490397975	3.18513490397975\\
-1.66666666666667	0.833333333333333	2.19942661996902	2.19942661996902\\
-1.66666666666667	1	1.48976154510575	1.48976154510575\\
-1.66666666666667	1.16666666666667	1.12822954047435	1.12822954047435\\
-1.66666666666667	1.33333333333333	1.34528756892585	1.34528756892585\\
-1.66666666666667	1.5	3.98499761106936	3.98499761106936\\
-1.66666666666667	1.66666666666667	9.86277789832878	9.86277789832878\\
-1.66666666666667	1.83333333333333	19.8903473320029	19.8903473320029\\
-1.66666666666667	2	36.1084289930423	36.1084289930423\\
-1.5	-2	-490.735001546795	-490.735001546795\\
-1.5	-1.83333333333333	-309.178148943492	-309.178148943492\\
-1.5	-1.66666666666667	-186.37528576108	-186.37528576108\\
-1.5	-1.5	-106.218353283967	-106.218353283967\\
-1.5	-1.33333333333333	-56.0126448719112	-56.0126448719112\\
-1.5	-1.16666666666667	-27.1996823788702	-27.1996823788702\\
-1.5	-1	-12.9839903021107	-12.9839903021107\\
-1.5	-0.833333333333333	-5.66820232171093	-5.66820232171093\\
-1.5	-0.666666666666667	-2.35030545309966	-2.35030545309966\\
-1.5	-0.5	-0.586442941338758	-0.586442941338758\\
-1.5	-0.333333333333333	0.617610372457747	0.617610372457747\\
-1.5	-0.166666666666667	1.48833881759745	1.48833881759745\\
-1.5	0	1.98769247087897	1.98769247087897\\
-1.5	0.166666666666667	2.07271436870004	2.07271436870004\\
-1.5	0.333333333333333	1.71862944009359	1.71862944009359\\
-1.5	0.5	1.14432803180495	1.14432803180495\\
-1.5	0.666666666666667	0.624499888707162	0.624499888707162\\
-1.5	0.833333333333333	0.238729728205337	0.238729728205337\\
-1.5	1	0.0398477704584536	0.0398477704584536\\
-1.5	1.16666666666667	0.166531208375309	0.166531208375309\\
-1.5	1.33333333333333	1.35869951757524	1.35869951757524\\
-1.5	1.5	4.08069578179833	4.08069578179833\\
-1.5	1.66666666666667	9.11618037837124	9.11618037837124\\
-1.5	1.83333333333333	17.8193886769481	17.8193886769481\\
-1.5	2	32.213275054003	32.213275054003\\
-1.33333333333333	-2	-394.42718955162	-394.42718955162\\
-1.33333333333333	-1.83333333333333	-245.99885279102	-245.99885279102\\
-1.33333333333333	-1.66666666666667	-146.421832395545	-146.421832395545\\
-1.33333333333333	-1.5	-82.3532919595946	-82.3532919595946\\
-1.33333333333333	-1.33333333333333	-43.0453701676366	-43.0453701676366\\
-1.33333333333333	-1.16666666666667	-20.3349103942965	-20.3349103942965\\
-1.33333333333333	-1	-8.80141998353798	-8.80141998353798\\
-1.33333333333333	-0.833333333333333	-3.31560661087624	-3.31560661087624\\
-1.33333333333333	-0.666666666666667	-1.5871486992677	-1.5871486992677\\
-1.33333333333333	-0.5	-0.885965676645044	-0.885965676645044\\
-1.33333333333333	-0.333333333333333	-0.340466883657845	-0.340466883657845\\
-1.33333333333333	-0.166666666666667	0.135816080931246	0.135816080931246\\
-1.33333333333333	0	0.444486037449819	0.444486037449819\\
-1.33333333333333	0.166666666666667	0.513077553768285	0.513077553768285\\
-1.33333333333333	0.333333333333333	0.335446713264925	0.335446713264925\\
-1.33333333333333	0.5	0.128417532865189	0.128417532865189\\
-1.33333333333333	0.666666666666667	-0.00607971018157955	-0.00607971018157955\\
-1.33333333333333	0.833333333333333	-0.0162832084401537	-0.0162832084401537\\
-1.33333333333333	1	0.0413886947832218	0.0413886947832218\\
-1.33333333333333	1.16666666666667	0.394003131694414	0.394003131694414\\
-1.33333333333333	1.33333333333333	1.46935914744609	1.46935914744609\\
-1.33333333333333	1.5	3.73647081426591	3.73647081426591\\
-1.33333333333333	1.66666666666667	8.03092242251029	8.03092242251029\\
-1.33333333333333	1.83333333333333	15.7055153328528	15.7055153328528\\
-1.33333333333333	2	28.8361544331395	28.8361544331395\\
-1.16666666666667	-2	-311.112965622826	-311.112965622826\\
-1.16666666666667	-1.83333333333333	-193.179704396051	-193.179704396051\\
-1.16666666666667	-1.66666666666667	-113.006897225493	-113.006897225493\\
-1.16666666666667	-1.5	-62.0827188122231	-62.0827188122231\\
-1.16666666666667	-1.33333333333333	-31.5522128245192	-31.5522128245192\\
-1.16666666666667	-1.16666666666667	-14.4003675628026	-14.4003675628026\\
-1.16666666666667	-1	-5.55161911762755	-5.55161911762755\\
-1.16666666666667	-0.833333333333333	-1.74596479438212	-1.74596479438212\\
-1.16666666666667	-0.666666666666667	-0.759563471015594	-0.759563471015594\\
-1.16666666666667	-0.5	-0.582802797433991	-0.582802797433991\\
-1.16666666666667	-0.333333333333333	-0.379027926473403	-0.379027926473403\\
-1.16666666666667	-0.166666666666667	-0.145309009335279	-0.145309009335279\\
-1.16666666666667	0	0.0270133224515199	0.0270133224515199\\
-1.16666666666667	0.166666666666667	0.0891051301646898	0.0891051301646898\\
-1.16666666666667	0.333333333333333	0.0853762952509867	0.0853762952509867\\
-1.16666666666667	0.5	0.0613640823939413	0.0613640823939413\\
-1.16666666666667	0.666666666666667	0.0654601563797474	0.0654601563797474\\
-1.16666666666667	0.833333333333333	0.0952641732296808	0.0952641732296808\\
-1.16666666666667	1	0.11976694351868	0.11976694351868\\
-1.16666666666667	1.16666666666667	0.429785567341452	0.429785567341452\\
-1.16666666666667	1.33333333333333	1.30204747376753	1.30204747376753\\
-1.16666666666667	1.5	3.19401807846878	3.19401807846878\\
-1.16666666666667	1.66666666666667	6.9503084286413	6.9503084286413\\
-1.16666666666667	1.83333333333333	13.9364456004984	13.9364456004984\\
-1.16666666666667	2	26.2409570256736	26.2409570256736\\
-1	-2	-219.314453830979	-219.314453830979\\
-1	-1.83333333333333	-141.307488098624	-141.307488098624\\
-1	-1.66666666666667	-83.7357637153283	-83.7357637153283\\
-1	-1.5	-44.310978463277	-44.310978463277\\
-1	-1.33333333333333	-21.5751829930939	-21.5751829930939\\
-1	-1.16666666666667	-9.41068505975817	-9.41068505975817\\
-1	-1	-3.47176967642999	-3.47176967642999\\
-1	-0.833333333333333	-0.938931734600099	-0.938931734600099\\
-1	-0.666666666666667	-0.179906432826499	-0.179906432826499\\
-1	-0.5	-0.103266739529496	-0.103266739529496\\
-1	-0.333333333333333	-0.0611797094507088	-0.0611797094507088\\
-1	-0.166666666666667	0.00233845222120388	0.00233845222120388\\
-1	0	0.0425839387153963	0.0425839387153963\\
-1	0.166666666666667	0.0596635144676632	0.0596635144676632\\
-1	0.333333333333333	0.068636484355636	0.068636484355636\\
-1	0.5	0.052938481102666	0.052938481102666\\
-1	0.666666666666667	0.0437084495621261	0.0437084495621261\\
-1	0.833333333333333	0.0446672773439954	0.0446672773439954\\
-1	1	0.0892055044552314	0.0892055044552314\\
-1	1.16666666666667	0.344171324845855	0.344171324845855\\
-1	1.33333333333333	1.04939472235122	1.04939472235122\\
-1	1.5	2.67877891240525	2.67877891240525\\
-1	1.66666666666667	6.09252178072296	6.09252178072296\\
-1	1.83333333333333	12.6893943763167	12.6893943763167\\
-1	2	24.6292792788812	24.6292792788812\\
-0.833333333333333	-2	-130.397033413212	-130.397033413212\\
-0.833333333333333	-1.83333333333333	-82.8469473722552	-82.8469473722552\\
-0.833333333333333	-1.66666666666667	-49.5530466417436	-49.5530466417436\\
-0.833333333333333	-1.5	-27.2747434020772	-27.2747434020772\\
-0.833333333333333	-1.33333333333333	-13.4315907366193	-13.4315907366193\\
-0.833333333333333	-1.16666666666667	-5.871095762782	-5.871095762782\\
-0.833333333333333	-1	-2.29870536978084	-2.29870536978084\\
-0.833333333333333	-0.833333333333333	-0.837315137962292	-0.837315137962292\\
-0.833333333333333	-0.666666666666667	-0.0610023075073972	-0.0610023075073972\\
-0.833333333333333	-0.5	0.158372261688506	0.158372261688506\\
-0.833333333333333	-0.333333333333333	0.119138304181711	0.119138304181711\\
-0.833333333333333	-0.166666666666667	0.0424495664518766	0.0424495664518766\\
-0.833333333333333	0	0.0074254596868183	0.0074254596868183\\
-0.833333333333333	0.166666666666667	0.0256744752470549	0.0256744752470549\\
-0.833333333333333	0.333333333333333	0.0343230943649486	0.0343230943649486\\
-0.833333333333333	0.5	0.0400773968811235	0.0400773968811235\\
-0.833333333333333	0.666666666666667	0.0303678188946733	0.0303678188946733\\
-0.833333333333333	0.833333333333333	0.0261604264846102	0.0261604264846102\\
-0.833333333333333	1	0.0693567609330355	0.0693567609330355\\
-0.833333333333333	1.16666666666667	0.24736673158028	0.24736673158028\\
-0.833333333333333	1.33333333333333	0.824662672779974	0.824662672779974\\
-0.833333333333333	1.5	2.30281504809991	2.30281504809991\\
-0.833333333333333	1.66666666666667	5.57823529623997	5.57823529623997\\
-0.833333333333333	1.83333333333333	12.1142416761576	12.1142416761576\\
-0.833333333333333	2	24.1866911999921	24.1866911999921\\
-0.666666666666667	-2	-64.8074709790725	-64.8074709790725\\
-0.666666666666667	-1.83333333333333	-41.5149079436972	-41.5149079436972\\
-0.666666666666667	-1.66666666666667	-25.0975123250734	-25.0975123250734\\
-0.666666666666667	-1.5	-14.136246218621	-14.136246218621\\
-0.666666666666667	-1.33333333333333	-7.29532793297624	-7.29532793297624\\
-0.666666666666667	-1.16666666666667	-3.40192131834611	-3.40192131834611\\
-0.666666666666667	-1	-1.4594876875698	-1.4594876875698\\
-0.666666666666667	-0.833333333333333	-0.587072954851501	-0.587072954851501\\
-0.666666666666667	-0.666666666666667	-0.0758744371989418	-0.0758744371989418\\
-0.666666666666667	-0.5	0.0520278457278398	0.0520278457278398\\
-0.666666666666667	-0.333333333333333	0.042254538293931	0.042254538293931\\
-0.666666666666667	-0.166666666666667	0.00330857420075475	0.00330857420075475\\
-0.666666666666667	0	0.00597888676745628	0.00597888676745628\\
-0.666666666666667	0.166666666666667	0.0126940426854142	0.0126940426854142\\
-0.666666666666667	0.333333333333333	0.0310849796275803	0.0310849796275803\\
-0.666666666666667	0.5	0.0229577566572396	0.0229577566572396\\
-0.666666666666667	0.666666666666667	0.0138461091981652	0.0138461091981652\\
-0.666666666666667	0.833333333333333	0.0201091591868431	0.0201091591868431\\
-0.666666666666667	1	0.0431274153320938	0.0431274153320938\\
-0.666666666666667	1.16666666666667	0.187077891554109	0.187077891554109\\
-0.666666666666667	1.33333333333333	0.679116820676426	0.679116820676426\\
-0.666666666666667	1.5	2.14324063528932	2.14324063528932\\
-0.666666666666667	1.66666666666667	5.50849500111093	5.50849500111093\\
-0.666666666666667	1.83333333333333	12.3343553966596	12.3343553966596\\
-0.666666666666667	2	25.0700229523623	25.0700229523623\\
-0.5	-2	-18.9267949526584	-18.9267949526584\\
-0.5	-1.83333333333333	-13.3763094355812	-13.3763094355812\\
-0.5	-1.66666666666667	-8.82606416554074	-8.82606416554074\\
-0.5	-1.5	-5.42930942069811	-5.42930942069811\\
-0.5	-1.33333333333333	-3.09910519692974	-3.09910519692974\\
-0.5	-1.16666666666667	-1.6288669402651	-1.6288669402651\\
-0.5	-1	-0.773879805303305	-0.773879805303305\\
-0.5	-0.833333333333333	-0.273279851391042	-0.273279851391042\\
-0.5	-0.666666666666667	-0.0354909256581537	-0.0354909256581537\\
-0.5	-0.5	0.0330293569599186	0.0330293569599186\\
-0.5	-0.333333333333333	0.0187954970775499	0.0187954970775499\\
-0.5	-0.166666666666667	0.000951225573527711	0.000951225573527711\\
-0.5	0	0.0152061285664857	0.0152061285664857\\
-0.5	0.166666666666667	0.0418091807071432	0.0418091807071432\\
-0.5	0.333333333333333	0.0320393747405116	0.0320393747405116\\
-0.5	0.5	0.0162464871505473	0.0162464871505473\\
-0.5	0.666666666666667	0.00714405446015717	0.00714405446015717\\
-0.5	0.833333333333333	0.00679931249047797	0.00679931249047797\\
-0.5	1	0.0423440452774169	0.0423440452774169\\
-0.5	1.16666666666667	0.154564130504981	0.154564130504981\\
-0.5	1.33333333333333	0.659705370256844	0.659705370256844\\
-0.5	1.5	2.2668228911592	2.2668228911592\\
-0.5	1.66666666666667	5.95686070194508	5.95686070194508\\
-0.5	1.83333333333333	13.4093885150946	13.4093885150946\\
-0.5	2	27.1844532297527	27.1844532297527\\
-0.333333333333333	-2	11.9908378235287	11.9908378235287\\
-0.333333333333333	-1.83333333333333	5.05992798206254	5.05992798206254\\
-0.333333333333333	-1.66666666666667	1.54450651097925	1.54450651097925\\
-0.333333333333333	-1.5	0.0020824493260676	0.0020824493260676\\
-0.333333333333333	-1.33333333333333	-0.487755898915521	-0.487755898915521\\
-0.333333333333333	-1.16666666666667	-0.478792225830619	-0.478792225830619\\
-0.333333333333333	-1	-0.283721809000901	-0.283721809000901\\
-0.333333333333333	-0.833333333333333	-0.0843896768117323	-0.0843896768117323\\
-0.333333333333333	-0.666666666666667	0.0304958979203202	0.0304958979203202\\
-0.333333333333333	-0.5	0.0372509033468626	0.0372509033468626\\
-0.333333333333333	-0.333333333333333	0.0167176270577581	0.0167176270577581\\
-0.333333333333333	-0.166666666666667	0.0224998061174691	0.0224998061174691\\
-0.333333333333333	0	0.0349639305444121	0.0349639305444121\\
-0.333333333333333	0.166666666666667	0.0347833062577219	0.0347833062577219\\
-0.333333333333333	0.333333333333333	0.0223551633147529	0.0223551633147529\\
-0.333333333333333	0.5	-0.00336849517501854	-0.00336849517501854\\
-0.333333333333333	0.666666666666667	-0.0094428259722622	-0.0094428259722622\\
-0.333333333333333	0.833333333333333	0.0220892712685237	0.0220892712685237\\
-0.333333333333333	1	0.0586442404617289	0.0586442404617289\\
-0.333333333333333	1.16666666666667	0.174901330375345	0.174901330375345\\
-0.333333333333333	1.33333333333333	0.791058506406194	0.791058506406194\\
-0.333333333333333	1.5	2.57425780607853	2.57425780607853\\
-0.333333333333333	1.66666666666667	6.50768503192358	6.50768503192358\\
-0.333333333333333	1.83333333333333	14.2422059558049	14.2422059558049\\
-0.333333333333333	2	28.2782124819063	28.2782124819063\\
-0.166666666666667	-2	31.9102501346529	31.9102501346529\\
-0.166666666666667	-1.83333333333333	16.5623315561125	16.5623315561125\\
-0.166666666666667	-1.66666666666667	7.8042249835467	7.8042249835467\\
-0.166666666666667	-1.5	3.18362250366593	3.18362250366593\\
-0.166666666666667	-1.33333333333333	1.01807384345373	1.01807384345373\\
-0.166666666666667	-1.16666666666667	0.191672391425373	0.191672391425373\\
-0.166666666666667	-1	-0.00380233780363631	-0.00380233780363631\\
-0.166666666666667	-0.833333333333333	0.0103180224775348	0.0103180224775348\\
-0.166666666666667	-0.666666666666667	0.0342953883740854	0.0342953883740854\\
-0.166666666666667	-0.5	0.0289159907303648	0.0289159907303648\\
-0.166666666666667	-0.333333333333333	-0.00453964831311452	-0.00453964831311452\\
-0.166666666666667	-0.166666666666667	0.0130108933829785	0.0130108933829785\\
-0.166666666666667	0	0.0228199990138675	0.0228199990138675\\
-0.166666666666667	0.166666666666667	0.0203637300362608	0.0203637300362608\\
-0.166666666666667	0.333333333333333	0.00462737971218267	0.00462737971218267\\
-0.166666666666667	0.5	0.00850034773558352	0.00850034773558352\\
-0.166666666666667	0.666666666666667	0.00574322331028555	0.00574322331028555\\
-0.166666666666667	0.833333333333333	0.0475203599242563	0.0475203599242563\\
-0.166666666666667	1	0.0535435334924425	0.0535435334924425\\
-0.166666666666667	1.16666666666667	0.21068342173715	0.21068342173715\\
-0.166666666666667	1.33333333333333	0.849009500916677	0.849009500916677\\
-0.166666666666667	1.5	2.60440381830764	2.60440381830764\\
-0.166666666666667	1.66666666666667	6.39339695400715	6.39339695400715\\
-0.166666666666667	1.83333333333333	13.7522266378233	13.7522266378233\\
-0.166666666666667	2	27.0088068554137	27.0088068554137\\
0	-2	43.9088504915913	43.9088504915913\\
0	-1.83333333333333	23.2242585935683	23.2242585935683\\
0	-1.66666666666667	11.284648055434	11.284648055434\\
0	-1.5	4.88279846375407	4.88279846375407\\
0	-1.33333333333333	1.78798222708345	1.78798222708345\\
0	-1.16666666666667	0.513585105038832	0.513585105038832\\
0	-1	0.109113460571215	0.109113460571215\\
0	-0.833333333333333	0.0394492698741748	0.0394492698741748\\
0	-0.666666666666667	0.0302665868875441	0.0302665868875441\\
0	-0.5	-0.0110311503940656	-0.0110311503940656\\
0	-0.333333333333333	-0.00930107347544822	-0.00930107347544822\\
0	-0.166666666666667	0.0038689486814814	0.0038689486814814\\
0	0	0.0101388526155762	0.0101388526155762\\
0	0.166666666666667	0.00703604668699733	0.00703604668699733\\
0	0.333333333333333	0.0211262906217696	0.0211262906217696\\
0	0.5	0.0268870501043081	0.0268870501043081\\
0	0.666666666666667	0.0188183385552289	0.0188183385552289\\
0	0.833333333333333	0.045239753054927	0.045239753054927\\
0	1	-0.00414863299703325	-0.00414863299703325\\
0	1.16666666666667	0.111525654300946	0.111525654300946\\
0	1.33333333333333	0.617444713109684	0.617444713109684\\
0	1.5	2.0102864479076	2.0102864479076\\
0	1.66666666666667	5.08006327848441	5.08006327848441\\
0	1.83333333333333	11.1311118500557	11.1311118500557\\
0	2	22.1551075278013	22.1551075278013\\
0.166666666666667	-2	50.4343551182668	50.4343551182668\\
0.166666666666667	-1.83333333333333	26.6743860456745	26.6743860456745\\
0.166666666666667	-1.66666666666667	12.9920932724294	12.9920932724294\\
0.166666666666667	-1.5	5.66680554892514	5.66680554892514\\
0.166666666666667	-1.33333333333333	2.11710517565903	2.11710517565903\\
0.166666666666667	-1.16666666666667	0.626508764872346	0.626508764872346\\
0.166666666666667	-1	0.12067114532654	0.12067114532654\\
0.166666666666667	-0.833333333333333	0.0423676950102973	0.0423676950102973\\
0.166666666666667	-0.666666666666667	0.0104124779528202	0.0104124779528202\\
0.166666666666667	-0.5	-0.00794506139917513	-0.00794506139917513\\
0.166666666666667	-0.333333333333333	-0.00291962864782012	-0.00291962864782012\\
0.166666666666667	-0.166666666666667	0.00586800022078189	0.00586800022078189\\
0.166666666666667	0	0.00847792027305401	0.00847792027305401\\
0.166666666666667	0.166666666666667	0.0177839989560762	0.0177839989560762\\
0.166666666666667	0.333333333333333	0.0288065953660898	0.0288065953660898\\
0.166666666666667	0.5	-0.00621305735259639	-0.00621305735259639\\
0.166666666666667	0.666666666666667	0.0108245457890059	0.0108245457890059\\
0.166666666666667	0.833333333333333	0.0104869441497965	0.0104869441497965\\
0.166666666666667	1	-0.130327496857845	-0.130327496857845\\
0.166666666666667	1.16666666666667	-0.21533896400502	-0.21533896400502\\
0.166666666666667	1.33333333333333	-0.125924483299813	-0.125924483299813\\
0.166666666666667	1.5	0.414774069586116	0.414774069586116\\
0.166666666666667	1.66666666666667	1.95641873230784	1.95641873230784\\
0.166666666666667	1.83333333333333	5.43018354385268	5.43018354385268\\
0.166666666666667	2	12.3097625094835	12.3097625094835\\
0.333333333333333	-2	53.33968169745	53.33968169745\\
0.333333333333333	-1.83333333333333	28.1122923416148	28.1122923416148\\
0.333333333333333	-1.66666666666667	13.6576786083087	13.6576786083087\\
0.333333333333333	-1.5	5.95327638116706	5.95327638116706\\
0.333333333333333	-1.33333333333333	2.2216660701283	2.2216660701283\\
0.333333333333333	-1.16666666666667	0.627434862175743	0.627434862175743\\
0.333333333333333	-1	0.0890586572535332	0.0890586572535332\\
0.333333333333333	-0.833333333333333	0.0416548751187235	0.0416548751187235\\
0.333333333333333	-0.666666666666667	0.0294644098814954	0.0294644098814954\\
0.333333333333333	-0.5	0.0100720105138726	0.0100720105138726\\
0.333333333333333	-0.333333333333333	0.00594375230492878	0.00594375230492878\\
0.333333333333333	-0.166666666666667	0.0139850438470126	0.0139850438470126\\
0.333333333333333	0	0.0223976842209865	0.0223976842209865\\
0.333333333333333	0.166666666666667	0.0156156573945003	0.0156156573945003\\
0.333333333333333	0.333333333333333	-0.0143702045194423	-0.0143702045194423\\
0.333333333333333	0.5	-0.00777983449531884	-0.00777983449531884\\
0.333333333333333	0.666666666666667	0.0122298672050908	0.0122298672050908\\
0.333333333333333	0.833333333333333	-0.0986874002233792	-0.0986874002233792\\
0.333333333333333	1	-0.397339044195835	-0.397339044195835\\
0.333333333333333	1.16666666666667	-0.904576334393802	-0.904576334393802\\
0.333333333333333	1.33333333333333	-1.65263943577145	-1.65263943577145\\
0.333333333333333	1.5	-2.63325579857641	-2.63325579857641\\
0.333333333333333	1.66666666666667	-3.69057474276042	-3.69057474276042\\
0.333333333333333	1.83333333333333	-4.43140847859752	-4.43140847859752\\
0.333333333333333	2	-4.10230574465772	-4.10230574465772\\
0.5	-2	53.810392925832	53.810392925832\\
0.5	-1.83333333333333	28.2250824055844	28.2250824055844\\
0.5	-1.66666666666667	13.7001376629341	13.7001376629341\\
0.5	-1.5	6.00256819557358	6.00256819557358\\
0.5	-1.33333333333333	2.26733620855345	2.26733620855345\\
0.5	-1.16666666666667	0.664214866581747	0.664214866581747\\
0.5	-1	0.0837903784450094	0.0837903784450094\\
0.5	-0.833333333333333	0.0422398955204819	0.0422398955204819\\
0.5	-0.666666666666667	0.0353689192438627	0.0353689192438627\\
0.5	-0.5	0.000253418091449149	0.000253418091449149\\
0.5	-0.333333333333333	0.0171181171489277	0.0171181171489277\\
0.5	-0.166666666666667	0.0371409750803782	0.0371409750803782\\
0.5	0	0.0279630047478066	0.0279630047478066\\
0.5	0.166666666666667	0.0154230500387975	0.0154230500387975\\
0.5	0.333333333333333	0.00432705097433476	0.00432705097433476\\
0.5	0.5	0.0109049104300765	0.0109049104300765\\
0.5	0.666666666666667	-0.062972219969454	-0.062972219969454\\
0.5	0.833333333333333	-0.296498156117491	-0.296498156117491\\
0.5	1	-0.890832704228987	-0.890832704228987\\
0.5	1.16666666666667	-2.12782138132938	-2.12782138132938\\
0.5	1.33333333333333	-4.24415946494221	-4.24415946494221\\
0.5	1.5	-7.62232754019421	-7.62232754019421\\
0.5	1.66666666666667	-12.6433261667606	-12.6433261667606\\
0.5	1.83333333333333	-19.6341430350024	-19.6341430350024\\
0.5	2	-28.7801365607544	-28.7801365607544\\
0.666666666666667	-2	53.3034350286168	53.3034350286168\\
0.666666666666667	-1.83333333333333	28.2484416532542	28.2484416532542\\
0.666666666666667	-1.66666666666667	13.9359032980802	13.9359032980802\\
0.666666666666667	-1.5	6.25389628031244	6.25389628031244\\
0.666666666666667	-1.33333333333333	2.4515873415123	2.4515873415123\\
0.666666666666667	-1.16666666666667	0.769569223138716	0.769569223138716\\
0.666666666666667	-1	0.133032279362994	0.133032279362994\\
0.666666666666667	-0.833333333333333	0.0274285874661035	0.0274285874661035\\
0.666666666666667	-0.666666666666667	0.00494430472621752	0.00494430472621752\\
0.666666666666667	-0.5	0.0119212545403741	0.0119212545403741\\
0.666666666666667	-0.333333333333333	0.0267327765670087	0.0267327765670087\\
0.666666666666667	-0.166666666666667	0.022957503208936	0.022957503208936\\
0.666666666666667	0	0.023232723948731	0.023232723948731\\
0.666666666666667	0.166666666666667	0.0110089147343501	0.0110089147343501\\
0.666666666666667	0.333333333333333	0.0250146320580214	0.0250146320580214\\
0.666666666666667	0.5	-0.000296147993063801	-0.000296147993063801\\
0.666666666666667	0.666666666666667	-0.151977148268689	-0.151977148268689\\
0.666666666666667	0.833333333333333	-0.590451730072269	-0.590451730072269\\
0.666666666666667	1	-1.67074607561341	-1.67074607561341\\
0.666666666666667	1.16666666666667	-4.02631608917396	-4.02631608917396\\
0.666666666666667	1.33333333333333	-8.14780790580525	-8.14780790580525\\
0.666666666666667	1.5	-14.9678776552929	-14.9678776552929\\
0.666666666666667	1.66666666666667	-25.5622531183832	-25.5622531183832\\
0.666666666666667	1.83333333333333	-41.1671710271444	-41.1671710271444\\
0.666666666666667	2	-63.1331992854513	-63.1331992854513\\
0.833333333333333	-2	53.9589346045472	53.9589346045472\\
0.833333333333333	-1.83333333333333	28.9411633598083	28.9411633598083\\
0.833333333333333	-1.66666666666667	14.4868433570794	14.4868433570794\\
0.833333333333333	-1.5	6.61637081787008	6.61637081787008\\
0.833333333333333	-1.33333333333333	2.65138626661872	2.65138626661872\\
0.833333333333333	-1.16666666666667	0.861629877983851	0.861629877983851\\
0.833333333333333	-1	0.18221268120663	0.18221268120663\\
0.833333333333333	-0.833333333333333	0.0196433052059974	0.0196433052059974\\
0.833333333333333	-0.666666666666667	0.00672485953123257	0.00672485953123257\\
0.833333333333333	-0.5	-0.000984771596110743	-0.000984771596110743\\
0.833333333333333	-0.333333333333333	-0.0106677338944979	-0.0106677338944979\\
0.833333333333333	-0.166666666666667	-0.00118287834313087	-0.00118287834313087\\
0.833333333333333	0	0.00861774657028733	0.00861774657028733\\
0.833333333333333	0.166666666666667	0.0401425648910321	0.0401425648910321\\
0.833333333333333	0.333333333333333	0.0879814151786843	0.0879814151786843\\
0.833333333333333	0.5	0.0847860176060096	0.0847860176060096\\
0.833333333333333	0.666666666666667	-0.122425512827067	-0.122425512827067\\
0.833333333333333	0.833333333333333	-0.856939965245745	-0.856939965245745\\
0.833333333333333	1	-2.70442935016446	-2.70442935016446\\
0.833333333333333	1.16666666666667	-6.51592446116143	-6.51592446116143\\
0.833333333333333	1.33333333333333	-13.3769936562383	-13.3769936562383\\
0.833333333333333	1.5	-24.8181118095802	-24.8181118095802\\
0.833333333333333	1.66666666666667	-42.7965060967308	-42.7965060967308\\
0.833333333333333	1.83333333333333	-69.6636237605203	-69.6636237605203\\
0.833333333333333	2	-108.18047261976	-108.18047261976\\
1	-2	55.1428908622873	55.1428908622873\\
1	-1.83333333333333	29.7498176808579	29.7498176808579\\
1	-1.66666666666667	14.9362258051973	14.9362258051973\\
1	-1.5	6.79031208791723	6.79031208791723\\
1	-1.33333333333333	2.65357886706905	2.65357886706905\\
1	-1.16666666666667	0.784048733982233	0.784048733982233\\
1	-1	0.0947825589175876	0.0947825589175876\\
1	-0.833333333333333	0.0893850805106068	0.0893850805106068\\
1	-0.666666666666667	0.0812987119711988	0.0812987119711988\\
1	-0.5	0.0620295530367891	0.0620295530367891\\
1	-0.333333333333333	0.0550864609763928	0.0550864609763928\\
1	-0.166666666666667	0.0402563192840432	0.0402563192840432\\
1	0	0.0413103870989691	0.0413103870989691\\
1	0.166666666666667	-0.0216353016377109	-0.0216353016377109\\
1	0.333333333333333	-0.161488934339856	-0.161488934339856\\
1	0.5	-0.266575083041392	-0.266575083041392\\
1	0.666666666666667	-0.2727786891982	-0.2727786891982\\
1	0.833333333333333	-0.969531245567182	-0.969531245567182\\
1	1	-3.65302813815921	-3.65302813815921\\
1	1.16666666666667	-9.10898955565085	-9.10898955565085\\
1	1.33333333333333	-19.1500842882101	-19.1500842882101\\
1	1.5	-36.1988836680526	-36.1988836680526\\
1	1.66666666666667	-63.406423348132	-63.406423348132\\
1	1.83333333333333	-104.641135792446	-104.641135792446\\
1	2	-164.274031583879	-164.274031583879\\
1.16666666666667	-2	55.9248487077935	55.9248487077935\\
1.16666666666667	-1.83333333333333	29.949953282103	29.949953282103\\
1.16666666666667	-1.66666666666667	14.7203820551711	14.7203820551711\\
1.16666666666667	-1.5	6.32982770469691	6.32982770469691\\
1.16666666666667	-1.33333333333333	2.09191871087732	2.09191871087732\\
1.16666666666667	-1.16666666666667	0.259766170926868	0.259766170926868\\
1.16666666666667	-1	0.40707528444164	0.40707528444164\\
1.16666666666667	-0.833333333333333	0.527581071679057	0.527581071679057\\
1.16666666666667	-0.666666666666667	0.479063379093643	0.479063379093643\\
1.16666666666667	-0.5	0.419923110197832	0.419923110197832\\
1.16666666666667	-0.333333333333333	0.39910137821903	0.39910137821903\\
1.16666666666667	-0.166666666666667	0.412610304862126	0.412610304862126\\
1.16666666666667	0	0.369552790004052	0.369552790004052\\
1.16666666666667	0.166666666666667	0.171862788328514	0.171862788328514\\
1.16666666666667	0.333333333333333	-0.122945626331721	-0.122945626331721\\
1.16666666666667	0.5	-0.395144176602295	-0.395144176602295\\
1.16666666666667	0.666666666666667	-0.581990532222703	-0.581990532222703\\
1.16666666666667	0.833333333333333	-1.27808448813818	-1.27808448813818\\
1.16666666666667	1	-4.23472828765207	-4.23472828765207\\
1.16666666666667	1.16666666666667	-11.7870849523188	-11.7870849523188\\
1.16666666666667	1.33333333333333	-25.5940574330792	-25.5940574330792\\
1.16666666666667	1.5	-49.0110602775465	-49.0110602775465\\
1.16666666666667	1.66666666666667	-86.4572158692562	-86.4572158692562\\
1.16666666666667	1.83333333333333	-143.554046469807	-143.554046469807\\
1.16666666666667	2	-227.180758932416	-227.180758932416\\
1.33333333333333	-2	55.1278617372028	55.1278617372028\\
1.33333333333333	-1.83333333333333	28.5724893195047	28.5724893195047\\
1.33333333333333	-1.66666666666667	13.0267249397368	13.0267249397368\\
1.33333333333333	-1.5	4.54315980174599	4.54315980174599\\
1.33333333333333	-1.33333333333333	0.92245181630656	0.92245181630656\\
1.33333333333333	-1.16666666666667	1.37819039103721	1.37819039103721\\
1.33333333333333	-1	1.92650902817062	1.92650902817062\\
1.33333333333333	-0.833333333333333	1.96668989724321	1.96668989724321\\
1.33333333333333	-0.666666666666667	1.87642909205752	1.87642909205752\\
1.33333333333333	-0.5	1.82979257446328	1.82979257446328\\
1.33333333333333	-0.333333333333333	1.86115662200948	1.86115662200948\\
1.33333333333333	-0.166666666666667	1.9039924717765	1.9039924717765\\
1.33333333333333	0	1.74082853650672	1.74082853650672\\
1.33333333333333	0.166666666666667	1.34191796745102	1.34191796745102\\
1.33333333333333	0.333333333333333	0.74227532130634	0.74227532130634\\
1.33333333333333	0.5	0.0476821397498099	0.0476821397498099\\
1.33333333333333	0.666666666666667	-0.744940656341219	-0.744940656341219\\
1.33333333333333	0.833333333333333	-2.14302815547545	-2.14302815547545\\
1.33333333333333	1	-5.42799628137948	-5.42799628137948\\
1.33333333333333	1.16666666666667	-13.7862563572422	-13.7862563572422\\
1.33333333333333	1.33333333333333	-32.2150096903524	-32.2150096903524\\
1.33333333333333	1.5	-63.2267738419168	-63.2267738419168\\
1.33333333333333	1.66666666666667	-112.55914844084	-112.55914844084\\
1.33333333333333	1.83333333333333	-187.631666570319	-187.631666570319\\
1.33333333333333	2	-297.678335131605	-297.678335131605\\
1.5	-2	51.2155988900447	51.2155988900447\\
1.5	-1.83333333333333	24.3086050279576	24.3086050279576\\
1.5	-1.66666666666667	8.84825035874316	8.84825035874316\\
1.5	-1.5	2.85765644756469	2.85765644756469\\
1.5	-1.33333333333333	3.72571724007103	3.72571724007103\\
1.5	-1.16666666666667	5.27122632823601	5.27122632823601\\
1.5	-1	5.68533656037008	5.68533656037008\\
1.5	-0.833333333333333	5.66782395174979	5.66782395174979\\
1.5	-0.666666666666667	5.6042111504629	5.6042111504629\\
1.5	-0.5	5.65164518313373	5.65164518313373\\
1.5	-0.333333333333333	5.81424833757525	5.81424833757525\\
1.5	-0.166666666666667	5.82345765235046	5.82345765235046\\
1.5	0	5.47701260025528	5.47701260025528\\
1.5	0.166666666666667	4.72976343403734	4.72976343403734\\
1.5	0.333333333333333	3.57675575015089	3.57675575015089\\
1.5	0.5	2.05902872654537	2.05902872654537\\
1.5	0.666666666666667	0.0654250178452198	0.0654250178452198\\
1.5	0.833333333333333	-2.96702513792931	-2.96702513792931\\
1.5	1	-7.98672122253447	-7.98672122253447\\
1.5	1.16666666666667	-16.1860931044842	-16.1860931044842\\
1.5	1.33333333333333	-37.6475423705678	-37.6475423705678\\
1.5	1.5	-77.7680222880654	-77.7680222880654\\
1.5	1.66666666666667	-141.118387760269	-141.118387760269\\
1.5	1.83333333333333	-236.93829388139	-236.93829388139\\
1.5	2	-376.925244830748	-376.925244830748\\
1.66666666666667	-2	42.3780500272925	42.3780500272925\\
1.66666666666667	-1.83333333333333	16.6882146626043	16.6882146626043\\
1.66666666666667	-1.66666666666667	7.46930138687514	7.46930138687514\\
1.66666666666667	-1.5	8.71906070375091	8.71906070375091\\
1.66666666666667	-1.33333333333333	12.1903826706006	12.1903826706006\\
1.66666666666667	-1.16666666666667	13.5251549931992	13.5251549931992\\
1.66666666666667	-1	13.8503250751405	13.8503250751405\\
1.66666666666667	-0.833333333333333	13.875920066574	13.875920066574\\
1.66666666666667	-0.666666666666667	13.9744043079673	13.9744043079673\\
1.66666666666667	-0.5	14.2952065308055	14.2952065308055\\
1.66666666666667	-0.333333333333333	14.6378651506711	14.6378651506711\\
1.66666666666667	-0.166666666666667	14.5660940527481	14.5660940527481\\
1.66666666666667	0	13.9254075958607	13.9254075958607\\
1.66666666666667	0.166666666666667	12.5912024389971	12.5912024389971\\
1.66666666666667	0.333333333333333	10.4719489941144	10.4719489941144\\
1.66666666666667	0.5	7.49420561450747	7.49420561450747\\
1.66666666666667	0.666666666666667	3.41891541746223	3.41891541746223\\
1.66666666666667	0.833333333333333	-2.441376065981	-2.441376065981\\
1.66666666666667	1	-10.9897847047968	-10.9897847047968\\
1.66666666666667	1.16666666666667	-21.7349129231411	-21.7349129231411\\
1.66666666666667	1.33333333333333	-42.1135081393626	-42.1135081393626\\
1.66666666666667	1.5	-90.3649825183723	-90.3649825183723\\
1.66666666666667	1.66666666666667	-170.241339390638	-170.241339390638\\
1.66666666666667	1.83333333333333	-290.224108503764	-290.224108503764\\
1.66666666666667	2	-464.480319963376	-464.480319963376\\
1.83333333333333	-2	31.7087667465134	31.7087667465134\\
1.83333333333333	-1.83333333333333	17.194880266119	17.194880266119\\
1.83333333333333	-1.66666666666667	18.4760627386283	18.4760627386283\\
1.83333333333333	-1.5	25.2759849237434	25.2759849237434\\
1.83333333333333	-1.33333333333333	28.4363329701675	28.4363329701675\\
1.83333333333333	-1.16666666666667	29.6090542734379	29.6090542734379\\
1.83333333333333	-1	29.9817791084488	29.9817791084488\\
1.83333333333333	-0.833333333333333	30.2453458458891	30.2453458458891\\
1.83333333333333	-0.666666666666667	30.7989203235768	30.7989203235768\\
1.83333333333333	-0.5	31.6279111376636	31.6279111376636\\
1.83333333333333	-0.333333333333333	32.1756584429962	32.1756584429962\\
1.83333333333333	-0.166666666666667	31.9859264291905	31.9859264291905\\
1.83333333333333	0	30.8506146480398	30.8506146480398\\
1.83333333333333	0.166666666666667	28.5336379466986	28.5336379466986\\
1.83333333333333	0.333333333333333	24.8168206121981	24.8168206121981\\
1.83333333333333	0.5	19.464861176904	19.464861176904\\
1.83333333333333	0.666666666666667	12.0457040677182	12.0457040677182\\
1.83333333333333	0.833333333333333	1.69918001085579	1.69918001085579\\
1.83333333333333	1	-12.5916032987198	-12.5916032987198\\
1.83333333333333	1.16666666666667	-29.866433409545	-29.866433409545\\
1.83333333333333	1.33333333333333	-51.6559023164246	-51.6559023164246\\
1.83333333333333	1.5	-99.8121717425231	-99.8121717425231\\
1.83333333333333	1.66666666666667	-196.377958416657	-196.377958416657\\
1.83333333333333	1.83333333333333	-344.428521354403	-344.428521354403\\
1.83333333333333	2	-558.134788374738	-558.134788374738\\
2	-2	35.8716797824539	35.8716797824539\\
2	-1.83333333333333	36.2937617276021	36.2937617276021\\
2	-1.66666666666667	48.4670141497096	48.4670141497096\\
2	-1.5	54.8769919964473	54.8769919964473\\
2	-1.33333333333333	57.7500865544645	57.7500865544645\\
2	-1.16666666666667	58.9330725349057	58.9330725349057\\
2	-1	59.5956273687403	59.5956273687403\\
2	-0.833333333333333	60.5049208696462	60.5049208696462\\
2	-0.666666666666667	61.9437821140473	61.9437821140473\\
2	-0.5	63.525008180521	63.525008180521\\
2	-0.333333333333333	64.3118759358152	64.3118759358152\\
2	-0.166666666666667	63.8956724128771	63.8956724128771\\
2	0	61.9975944615609	61.9975944615609\\
2	0.166666666666667	58.1377162240677	58.1377162240677\\
2	0.333333333333333	51.8848092734573	51.8848092734573\\
2	0.5	42.8375882425218	42.8375882425218\\
2	0.666666666666667	30.3284877664508	30.3284877664508\\
2	0.833333333333333	13.1195867792475	13.1195867792475\\
2	1	-10.0643436320384	-10.0643436320384\\
2	1.16666666666667	-37.9076458363961	-37.9076458363961\\
2	1.33333333333333	-66.5355381126802	-66.5355381126802\\
2	1.5	-113.945562784635	-113.945562784635\\
2	1.66666666666667	-216.373805761883	-216.373805761883\\
2	1.83333333333333	-394.233582664751	-394.233582664751\\
2	2	-653.199139432757	-653.199139432757\\
};
\end{axis}
\end{tikzpicture}%